\documentclass[12pt]{amsart}
\usepackage{amssymb,amsmath,mathrsfs}
\usepackage[colorlinks=true,urlcolor=blue,
citecolor=red,linkcolor=blue,linktocpage,pdfpagelabels,
bookmarksnumbered,bookmarksopen]{hyperref}
\usepackage[active]{srcltx}
\usepackage{verbatim}
\usepackage{epsfig,graphicx,color,mathrsfs}
\usepackage{graphicx}
\usepackage{amsmath,amssymb,amsthm,amsfonts}
\usepackage{amssymb}
\usepackage[english]{babel}
\usepackage[left=2.7cm,right=2.7cm,top=3cm,bottom=3cm]{geometry}

\newcommand{\R}{{\mathbb R}}

\def \d {\mathrm{d}}

\def \de {\partial}
\def \e {\varepsilon}
\newtheorem{thm}{Theorem}[section]

\newtheorem{lem}[thm]{Lemma}

\newtheorem{rem}[thm]{Remark}
\newtheorem{defn}[thm]{Definition}

\numberwithin{equation}{section}
\makeatletter
\@namedef{subjclassname@2020}{\textup{2020} Mathematics Subject Classification}
\makeatother
\begin{document}

\title[]{Brezis-Nirenberg-type results for the anisotropic $p$-Laplacian}

\author[S.\,Biagi]{Stefano Biagi}
\author[F.\,Esposito]{Francesco Esposito}
\author[A.\,Roncoroni]{Alberto Roncoroni}
\author[E.\,Vecchi]{Eugenio Vecchi}

\address[S.\,Biagi]{Politecnico di Milano - Dipartimento di Matematica
\newline\indent
Via Bonardi 9, 20133 Milano, Italy}
\email{stefano.biagi@polimi.it}

\address[F.\,Esposito]{Dipartimento di Matematica e Informatica, Universit\`a della Calabria
\newline\indent
Ponte Pietro Bucci 31B, 87036 Arcavacata di Rende, Cosenza, Italy}
\email{francesco.esposito@unical.it}

\address[A.\,Roncoroni]{Politecnico di Milano - Dipartimento di Matematica
\newline\indent
Via Bonardi 9, 20133 Milano, Italy}
\email{alberto.roncoroni@polimi.it}

\address[E.\,Vecchi]{Dipartimento di Matematica, Universit\`a di Bologna
\newline\indent
Piazza di Porta San Donato 5, 40126 Bologna, Italy}
\email{eugenio.vecchi2@unibo.it}



\thanks{S.B. and E.V. are partially supported by PRIN project 2022R537CS {\em ``$NO^3$ - Nodal Optimization, NOnlinear elliptic equations, NOnlocal geometric problems, with a focus on regularity''}. F.E. is  partially supported by PRIN project P2022YFAJH (Italy): 
{\em Linear and Nonlinear PDE's: New directions and Applications.}  A.R. is partially supported by PRIN project 20225J97H5 (Italy): {\em Differential geometric aspects of manifolds via global analysis} and by the INdAM-GNSAGA project {\em Canonical structures on manifolds}.  F.E. and E.V. are members of GNAMPA-INdAM. F.E. and E.V. are partially supported by the INdAM-GNAMPA project {\em “Esistenza, unicit\`a e regolarit\`a di soluzioni per problemi singolari”.} }

\begin{abstract}
In this paper we consider a quasilinear elliptic and critical problem with Dirichlet boundary conditions in 
presence of the anisotropic $p$-Laplacian. The critical exponent is the usual $p^{\star}$ such that the embedding 
$W^{1,p}_{0}(\Omega) \subset L^{p^{\star}}(\Omega)$ is not compact. We prove the existence of a weak positive 
solution in presence of both a $p$-linear and a $p$-superlinear perturbation. In doing this, we have to perform 
several precise estimates of the anisotropic Aubin-Talenti functions which can be of interest for further 
problems. The results we prove are a natural generalization to the anisotropic setting of the classical ones by 
Brezis-Nirenberg \cite{BN}.
\end{abstract}

\maketitle


\section{Introduction}\label{intro}

The aim of this paper is to study the following quasilinear elliptic problem
\begin{equation}\label{eq:Problem}
\left\{\begin{array}{rl}
-\Delta^{H}_{p}u = u^{p^{\star}-1} + \lambda u^{q-1} & \textrm{in   } \Omega,\\
u >0 & \textrm{in } \Omega,\\
u=0 & \textrm{on } \partial \Omega.
\end{array}\right.
\end{equation}
Here, $\Omega$ is a bounded smooth domain of $\R^n$, $p \in (1,+\infty)$, 
$$
p^{\star}:= \frac{np}{n-p}
$$ 
(for $n > p$) denotes the critical Sobolev exponent, $p\leq q<p^\star$, and
\begin{equation*}
-\Delta^{H}_{p}u := -\mathrm{div} \left(H(\nabla u)^{p-1}\nabla H(\nabla u)\right)
\end{equation*}
\noindent is the anisotropic $p$-Laplacian, where $H:\mathbb{R}^{n}\to \mathbb{R}$ 
satisfies the following properties:
\begin{itemize}
\item[$(i)$] $H$ is convex;	
\item[$(ii)$] $H$ is positively 1-homogeneous (i.e. $H(t\xi)=t H(\xi)$, 
for every $\xi \in \mathbb{R}^n$ and $t>0$);
\item[$(iii)$] $H(\xi)>0$ for every $\xi \in \mathbb{R}^n$;
\item[$(iv)$] $H \in C^{2}(\mathbb{R}^{n}\setminus \{0\})$;
\item[$(v)$] the set (Wulff shape or Alexandrov body) $B_1^H:=\{x \in \R^n \ : \ H(x)<1\}$ is uniformly convex.
\footnote{We recall that a set in $\R^n$ is said to be uniformly 
convex if and only if all the principal curvature of its boundary are bounded away from $0$.}
\end{itemize}
We will call $H$ a {\it norm} even though we do not ask it to be symmetric. Obviously, a distinguished example of
a function $H$ satisfying 
properties $(i)$-to-$(v)$ is the Euclidean norm $|\cdot|$; in this case, the
operator $-\Delta^{H}_{p}$ boils down to the usual $p$-Laplace operator, indeed
$$-\Delta^H_p u = -\mathrm{div}\Big(|\nabla u|^{p-1}\frac{\nabla u}{|\nabla u|}\Big) = -\Delta_p u.$$

\medskip

Before entering into the details of the notion of solution to \eqref{eq:Problem} and the related results, let us 
describe the main features of the Euclidean version of \eqref{eq:Problem}, i.e. with $-\Delta_{p}^H = -\Delta_p$.
The first appearance of this problem dates back to \cite{BN}, it concerns the case $p=2$ (i.e. the leading 
operator is $-\Delta$), and it is nowadays addressed as {\it Brezis-Nirenberg problem}. In this case 
\eqref{eq:Problem} reads as 
 \begin{equation}\label{eq:Problem_BN}
\left\{\begin{array}{rl}
-\Delta u = u^{2^{\star}-1} + \lambda u^{q-1} & \textrm{in   } \Omega,\\
u >0 & \textrm{in } \Omega,\\
u=0 & \textrm{on } \partial \Omega.
\end{array}\right.
\end{equation}
\noindent Here $2^{\star} := 2n/(n-2)$ denotes the critical Sobolev exponent (at least for $n\geq 3$) and $\Omega$ 
is as before.
The problem \eqref{eq:Problem_BN}, as well as \eqref{eq:Problem}, has a variational structure in the sense that it 
can be seen as the Euler-Lagrange equation of the functional $F_{\lambda}:H^{1}_{0}(\Omega)\to \mathbb{R}$ defined 
as
\begin{equation*}
  F_{\lambda}(u):= \dfrac{1}{2}\int_{\Omega}|\nabla u|^2 \, dx -
   \dfrac{\lambda}{q}\int_{\Omega}(u^+)^q\, dx - \dfrac{1}{2^{\star}}\int_{\Omega}(u^+)^{2^{\star}}\, dx, 
\end{equation*}
where $u^+:=\max\{0,u\}$. In particular, critical points of the above functional turn out to be weak solutions of
\eqref{eq:Problem_BN}, and this opens the way to the use of variational methods and critical point theory. As it 
is well known, the first obstruction which prevents a direct variational approach is given by the lack of 
compactness caused by the presence of the critical term. 
Looking more at the PDE level, the reason behind the introduction of a second power-type term is related to the 
fact that purely critical problem (i.e. when $\lambda=0$) does not admit a non-trivial solution in any bounded 
star-shaped domain as a consequence of the well-known Pohozaev identity. Actually, this is still the case, and for 
a very similar argument, if one consider $\lambda <0$. All of this led Brezis and Nirenberg to ask whether a 
subcritical perturbation, even a linear one, could allow to gain back the required compactness yielding the 
existence of mountain pass solutions. The answer is positive and, possibly more surprising, exhibits a different 
behaviour depending on the dimension $n$ of the ambient space $\mathbb{R}^{n}$. We summarize it here below 
distinguishing between the linear perturbation (i.e. $q=2$) and the superlinear one (i.e. $q>2$). In what follows, 
$\lambda_1$ denotes the first Dirichlet eigenvalue of $-\Delta$.\\

\noindent \textbf{Case $q=2$}:
\begin{itemize}
\item if $n=3$ and $\Omega$ is a ball, problem \eqref{eq:Problem_BN} admits a solution 
for $\lambda \in (\lambda_1 /4, \lambda_1)$. See \cite[Theorem 1.2]{BN};
\item if $n\geq 4$, problem \eqref{eq:Problem_BN} admits a solution for every $\lambda \in (0,\lambda_1)$, and it 
does not admit any for $\lambda \not\in (0,\lambda_1)$. See \cite[Theorem 1.1]{BN}.
\end{itemize}
A similar distinction has to be made also when dealing with superlinear perturbations. Recall that when $n=3$, the 
critical exponent $2^{\star}=6$.\\
\noindent \textbf{Case $2<q<2^{\star}$}:
\begin{itemize}
\item if $n=3$ and $4<q<6$ problem \eqref{eq:Problem_BN} admits a solution for every $\lambda >0$, 
 while for $2<q
 \leq 4$ problem \eqref{eq:Problem_BN} admits a solution for large enough $\lambda$. 
 See \cite[Corollary 2.3 and 2.4]{BN};
\item if $n\geq 4$, problem \eqref{eq:Problem_BN} admits a solution for every $\lambda >0$. 
 See \cite[Corollary 2.1 and 2.2]{BN}.
\end{itemize}
Going back to the core of the proof in the case $q=2$, the crucial result which allows to get the conclusions 
listed before is the following one: calling $S$ the best Sobolev constant of the embedding 
$H^{1}_{0}(\Omega)\subset L^{2^{\star}}(\Omega)$ and
\begin{equation*}
S_{\lambda}:= \inf \left\{F_{\lambda}(u): u \in H^{1}_{0}(\Omega), \|u\|_{L^{2^{\star}}(\Omega)}=1 \right\},
\end{equation*}
\noindent if $S_{\lambda} < S$ then $S_{\lambda}$ is achieved (see \cite[Lemma 1.2]{BN}). To prove that this is actually the case, since $S$ is known to be realized in the whole of $\mathbb{R}^{n}$ 
by the so called Aubin-Talenti functions
\begin{equation*}
U_{\varepsilon}(x) := \left( \dfrac{C_{\varepsilon}}{\varepsilon^{2} + |x|^2}\right)^{\tfrac{n-2}{2}}, 
 \quad \varepsilon >0,
\end{equation*} 
\noindent the fact that $S_{\lambda} < S$ is obtained by careful estimates of $F_{\lambda}(V_\varepsilon)$ in 
terms of $\varepsilon$, where $V_{\varepsilon}$ is a suitable modification of $U_\varepsilon$ and this causes the 
mentioned bounds on $\lambda$.\vspace*{0.1cm}

The superlinear case relies instead on an application of the Ambrosetti-Rabinowitz mountain pass theorem. 
Therefore, apart from proving that the functional $F_{\lambda}$ satisfies the proper geometric structure, one has 
to exhibit an explicit path whose energy level (w.r.t. $F_\lambda$) is below a certain threshold given by $
\tfrac{1}{n}S^{n/2}$. Once again, a major role is played by careful estimates of suitable modifications of the 
Aubin-Talenti functions (thus leading to the appearance of $S$ in the estimate). We stress that this {mountain 
pass approach} can be performed in the linear case $q=2$ as well. This for what concerns the case of $-\Delta$ 
acting as leading operator. \medskip

The problem has then been generalized to various other operators such as the 
$p$-La\-pla\-cian, see 
\cite{AngEsp,AG,GAP,GAP2}. The list of other generalizations would be too 
long and would bring us too far from the purposes of the present paper.
As before, let us write down the problem we are describing when the leading operator is the $p$-Laplacian:
 \begin{equation}\label{eq:Problem_p_Lap}
\left\{\begin{array}{rl}
-\Delta_p u = u^{p^{\star}-1} + \lambda u^{q-1} & \textrm{in   } \Omega,\\
u >0 & \textrm{in } \Omega,\\
u=0 & \textrm{on } \partial \Omega.
\end{array}\right.
\end{equation}
\noindent Here $p^{\star} := np/(n-p)$ denotes the critical Sobolev exponent 
(at least for $n> p$) and $\Omega$ is as before.
The known result we are interested in can be summarized as follows:
\medskip

\noindent {\bf 1)\,\,Case $q=p$:}
\begin{itemize}
\item (See \cite[Theorem 3]{AG}): if $1<p<n<p^2$, then problem \eqref{eq:Problem_p_Lap} admits a solution for
 $$\lambda \in (\lambda_{1,p}-S_p |
\Omega|^{-p/n},\lambda_{1,p});$$ 
\item (See \cite[Theorem 7.4]{GAP}): if $n\geq p^2$ then problem \eqref{eq:Problem_p_Lap} admits a solution for 
$$\lambda \in (0,\lambda_{1,p}).$$
\end{itemize}
\noindent Here, in analogy with what we wrote above, $\lambda_{1,p}$ denotes the first variational Dirichlet 
eigenvalue of $-\Delta_p$ and $S_p$ the best Sobolev constant of the embedding 
$W^{1,p}_{0}(\Omega)\subset 
L^{p^{\star}}(\Omega)$. As for the Laplacian $-\Delta$, there are results also in case of $p$-superlinear 
perturbations.
\medskip

\noindent {\bf 2)\,\,Case $p<q<p^{\star}$:}
\begin{itemize}
\item (See \cite[Theorem 3.3]{GAP2}):  problem \eqref{eq:Problem_p_Lap} admits a solution for 
every $\lambda >0$, provided that $n$ is sufficiently large, namely
$$n > \tfrac{p[q(p-1)+p]}{q(p-1)+p-p(p-1)};$$

\item (See \cite[Theorem 3.2]{GAP2}):  there exists some $\lambda_0>0$ such that problem 
\eqref{eq:Problem_p_Lap} admits a so\-lu\-tion for every $\lambda \geq \lambda_0$.
provided that
$$n\leq \tfrac{p[q(p-1)+p]}{q(p-1)+p-p(p-1)}.$$
\end{itemize}
As for the Laplacian case, the crucial steps involve once again careful estimates of the appropriate Aubin-Talenti 
functions combined with suitable technical modifications due to the fact that the Sobolev space $W^{1,p}_{0}
(\Omega)$ is not more a Hilbert space. Since we want to preserve the readability, we avoid further details 
concerning this.
\medskip 
 
Our main goal in the present paper is to consider the anisotropic Brezis-Nirenberg problem \eqref{eq:Problem},  where the leading operator is the so called anisotropic $p$-Laplacian. This operator is an important generalization of the classical $p$-Laplacian operator that extends its applications to anisotropic settings by incorporating the geometry of Finsler spaces. This operator arises in a variety of contexts where anisotropy plays a critical role, such as image processing \cite{EsOs, PeMa}, and material science \cite{CaHo, Gu}. Its significance lies in its ability to model diffusion and transport phenomena in systems where isotropic assumptions fail, providing a more accurate representation of the underlying processes \cite{Ohata}. Its role in understanding anisotropic mean curvature flows and evolution equations has been extensively studied, with connections to convex geometry and relative geometry frameworks (see, e.g., \cite{BePa} and references therein).
We mention that in the last decades, qualitative properties of solutions to elliptic problems involving the anisotropic $p$-Laplace operator have been investigated, a partial list of papers and topics includes \cite{BiCi, Bi-Ci-Sa, CiSal, CiSal2,CiGr,CiLi} in which 
overdetermined problems are studied, \cite{CCR,CFR,CiLi2,CoFaVa,CoFaVa2,DiPoVa,EMSV,ERSV, FaLi1, FaScVu, WangXia} in which symmetry properties and Liouville-type theorems are studied and \cite{BFK,CRS,DeMaMiNe,FMP} in  which isoperimetric inequalities and CMC hypersurfaces are studied. As it should be clear after the brief review of the Euclidean results, if we want to follow the fruitful scheme described before we have to work with the proper replacement of the Aubin-Talenti functions. In the anisotropic case these have been explicitly written down 
in \cite{CFR}, see \eqref{eq:HTalentiane} for the precise analytic expression.
One of the main technical difficulties in estimating it as we need, is given by the non-Euclidean structure of the anisotropic norm. 
\medskip

Let us now introduce the notion of {\it solution of problem \eqref{eq:Problem}} we will work with.
\begin{defn}\label{def:weak_sol}
We say that a function $u \in W^{1,p}_{0}(\Omega)$ is a weak solution of
problem \eqref{eq:Problem} if it satisfies the following
properties:
\begin{itemize}
 \item[i)] $u > 0$ a.e.\,on $\Omega$;
 \item[ii)] for every test function $v\in W_0^{1,p}(\Omega)$ we have
 \begin{equation} \label{eq:WeakSolbyParts}
   \int_{\Omega}H(\nabla u)^{p-1}\langle \nabla H(\nabla u),\nabla v\rangle \, dx - 
   \lambda \int_{\Omega}u^{q-1}v \, dx - \int_{\Omega}u^{p^\star-1}v\, dx=0.
 \end{equation}
\end{itemize}
\end{defn}
Due to its relevance in the sequel, we explicitly highlight that
problem \eqref{eq:Problem} has a \emph{good variational structure}
(as in the classical case $H(\xi) = |\xi|$ described above), that is,
the weak solutions of this problem (in the sense of Definition \ref{def:weak_sol})
are precisely the non-zero critical points of the $C^1$ functional
$\mathcal{J}_{q, \lambda}:W_0^{1,p}(\Omega)\to\R$ defined by
\begin{equation} \label{eq:defFunctionalJ}
 \mathcal{J}_{q, \lambda}(u)=\frac{1}{p} \int_{\Omega}H(\nabla u)^p \, dx - \frac{\lambda}{q} \int_{\Omega}(u^+)^q \, dx -  \frac{1}{p^\star} \int_{\Omega}(u^+)^{p^\star} \, dx.
\end{equation}
Indeed, if $u\in W_0^{1,p}(\Omega)$ is a non-zero critical point of $\mathcal{J}_{q, \lambda}$,
we clearly have that $u$ satisfies the associated Euler-Largange equation
$\mathcal{J}_{q, \lambda}'(u)[v]=0$ for every $v\in W_0^{1,p}(\Omega)$, that is,
$$
 \int_{\Omega}H(\nabla u)^{p-1}\langle \nabla H(\nabla u),\nabla v\rangle \, dx - \lambda \int_{\Omega}(u^{+})^{q-1}v \, dx - \int_{\Omega}(u^{+})^{p^\star-1}v\, dx=0;$$
 in particular, we derive that $-\Delta_p^H u\geq 0$ \emph{in the weak sense on $\Omega$}.
 We can then apply both the Weak and Strong Maximum Principle
 for $-\Delta_p^H$ (see \cite[Theorem 5.3.1]{PSbook}), obtaining that $u > 0$ a.e.\,in $\Omega$;
 thus, $u^+ = u$ in $\Omega$ and $u$ is a weak solution of problem \eqref{eq:Problem}.
\medskip

With Definition \ref{def:weak_sol} at hand, 
we are now ready to start the description of the main results of the paper,
which are essentially of two types: \emph{existence results} (to which
we are mainly interested, re\-pre\-sen\-ting the real core of
the manuscript) and \emph{non-existence results}.

The first existence result 
concerns the case of $p$-linear perturbations 
(that is, $q=p$) when $n\geq p^{2}$, which can be considered as the analogous of the {\it high-dimensional case} in \cite{BN}. In what follows, $\lambda_1^H(\Omega)$ denotes the first eigenvalue of $-\Delta_p^H$ (see \eqref{eq:firstEigenvalue} below).
\begin{thm}\label{thm:BN-style}
Let $\Omega \subset \mathbb{R}^{n}$ be an open and bounded set with smooth boundary. Let $1<p<p^2 \leq n$ and let $\lambda \in (0, \lambda^{H}_{1}(\Omega))$, then \eqref{eq:Problem} admits a solution.
\end{thm}
The second one still concerns the case $q=p$ but when $1<p<n<p^2$, which is the analogous of the {\it low dimensional case} in \cite{BN}.
\begin{thm}\label{thm_AG}
Let $\Lambda=\mathcal{S}_{H} \, |\Omega|^{-p/n}$, assume $1<p<n<p^2$ and $\lambda\in(\lambda_1^H(\Omega)-\Lambda,\lambda_1^H(\Omega))$. Then \eqref{eq:Problem} admits a nontrivial solution.
\end{thm}
The last existence result deals with the case of
$p$-superlinear perturbations (i.e. $q>p$). To properly state it, let us introduce
a constant which already appeared before, and that will provide the proper bound between what we can call {\it high} and {\it low} dimensional cases:
\begin{equation*}
\kappa_{p,q} = \frac{p[q(p-1)+p]}{q(p-1)+p-p(p-1)}.
\end{equation*}
\begin{thm}\label{thm:PertNonlinear}
	Let $\Omega \subset \mathbb{R}^{n}$ be an open and bounded set with smooth enough 
	boundary $\partial\Omega$ and let $1<p< n$, $p<q<p^\star$. Then the following holds:
	\begin{itemize}
	\item if $n>k_{p,q}$ the problem \eqref{eq:Problem} admits a solution for every $\lambda >0$;
	\item if $1<p< n \leq k_{p,q}$, there exists a positive number $\lambda_0>0$ such that problem \eqref{eq:Problem} admits a solution for every $\lambda \geq \lambda_0$. 
	\end{itemize}
\end{thm}
It is worth noting that, since $k_{p,p}=p^2$, the dichotomy
in Theorem \ref{thm:PertNonlinear} is
coherent with the threshold found when $q=p$; moreover, we point out that
\begin{equation*}
p<\kappa_{p,q}< p^2, \quad \textrm{for every } q\in (p,p^{\star}).
\end{equation*}
The constant $\kappa_{p,q}$ already appears implicitly in \cite{AG} (see condition (7) in there).
\vspace*{0.1cm}

As already highlighted, the crucial results needed to prove all the existence results listed above are the following:
\begin{itemize}
\item precise estimates of a suitable modification of the anisotropic Aubin-Talenti functions (Lemma \ref{lem:talentiane});
\item geometric structure of the functional $\mathcal{J}_{q,\lambda}$ (Lemma \ref{lem:MPGeometry});
\item convergence of Palais-Smale sequences at the proper energy-level to weak solutions of \eqref{eq:Problem} (Lemma \ref{lem:PSseq});
\item find explicit paths below a certain energy-level depending on the best anisotropic Sobolev constant $\mathcal{S}_{H}$ defined in \eqref{eq:AnisotropicSobolev} (Lemma \ref{lem:Path_q=p} for the case $q=p$ and $n>p^{2}$, Lemma \ref{lem:PathLow} or the case $q=p$ and $p<n<p^{2}$ and Lemma \ref{lem:Path} when $p<q<p^{\star}$).
\end{itemize}

Together with the existence results just stated, we
also establish the following \emph{non-exi\-stence} result.
In order to avoid regularity issues, we 
limit ourselves to consider the case
of regular solutions in star-shaped domains; this
allows us to exploit the anisotropic Pohozaev identity recently proved in \cite{MS} 
(see also \cite[Lemma 4.2]{CiSal}).
\begin{thm}\label{thm:Non_Existence}
Let $1<p<n$, and let $\Omega\subset\mathbb{R}^n$ be a bounded and star-shaped domain with smooth boundary.
If $\lambda\leq 0$, then \emph{there do not exist regular weak solutions} $u\in C^1(\overline{\Omega})$
of
problem \eqref{eq:Problem} with $q = p$, that is,
\begin{equation}\label{eq:Problem_general}
\left\{\begin{array}{rl}
-\Delta^{H}_{p}u = u^{p^{\star}-1} + \lambda u^{p-1} & \textrm{in   } \Omega,\\
u >0 & \textrm{in } \Omega,\\
u=0 & \textrm{on } \partial \Omega.
\end{array}\right.
\end{equation}
\end{thm}
Notice that Theorem \ref{thm:Non_Existence} somehow justifies the choice $\lambda > 0$
in our existence results
(at least when regular solutions and star-shaped domains are involved).
\medskip


\medskip

\noindent The paper is organized as follows:

\begin{itemize}
	\item In Section \ref{sec:Prel} we recall some notions about Finsler geometry, 
	and we prove some technical lemmas that will be crucial in the proof of the main results. In particular, 
	inspired by the celebrated paper of Brezis-Nirenberg, we prove some preliminary
	 estimates of suitable truncation of the Aubin-Talenti bubbles.

	\item Using all the preliminary results established/recalled
	in Section \ref{sec:Prel}, we provide
	in Section \ref{sec:Proof_BN} the proofs
	our \emph{existence results}, namely Theorems 
	 \ref{thm:BN-style}, \ref{thm_AG} and \ref{thm:PertNonlinear}.
	
	\item Finally, in Section \ref{sec:NonExist}
	we prove nonexistence result in Theorem \ref{thm:Non_Existence}. 
\end{itemize}

\section{Preliminaries}\label{sec:Prel}
The aim of this first section is to collect all the
notions and the preliminary results
which will be repeatedly used throughout the rest of the paper. 
Due to the number of topics we will discuss,
for the sake of readability we split our presentation into
several paragraphs.
\medskip

\noindent\textbf{i)\,\,The anisotropic geometry.}
In view of its (major) role in our results,
we begin this section with an overview of
the \emph{anisotropic geometry} induced on $\R^n$ by the norm $H$. 
\medskip

To begin with, 
we notice that assumptions $(i)$-to-$(iv)$ made on $H$ imply some simple consequences that we collect in the following Lemma for a future reference.

\begin{lem}\label{lem:propHeasy}
Let $H:\mathbb{R}^{n}\to \mathbb{R}$ be such that $(i)$-to-$(iv)$ hold. Then
\begin{enumerate}
  \item[{1)}] The function $H$ is \emph{globally equivalent} to the Euclidean norm $|\cdot|$, that is,
  there exists a universal constant $\nu\geq 1$ such that
  \begin{equation} \label{eq:HequivalentEu}
   \nu^{-1}|x|\leq H(x)\leq \nu|x|\quad\text{for every $x\in\R^n$}.
  \end{equation}
    
  \item[2)] The function $H$ satisfies the classical triangle inequality, that is,
  \begin{equation}\label{eq:triagnleH}
   H(x+y)\leq H(x)+H(y)\quad \text{for every $x,y\in\R^n$}.
  \end{equation}
  
   \item[3)] The function $H$ is \emph{globally Lipschitz-continuous} in $\R^n$.
   
\item[4)] The function $H$ satisfies the following Euler-type identity:
  \begin{equation} \label{eq:EulertypeH}
   \langle x,\nabla H(x)\rangle = H(x)\quad\text{for every $x\in\R^n\setminus\{0\}$}.
  \end{equation}
  \item[5)] There exists a universal constant $\vartheta\geq 1$ such that
  \begin{equation} \label{eq:nablaHBd}
   \vartheta^{-1}\leq |\nabla H(x)|\leq \vartheta\quad\text{for every $x\in\R^{n}\setminus\{0\}$}.
  \end{equation}
 \end{enumerate}     
  \begin{proof}
  1) follows from \emph{(iii)}-\emph{(iv)}, indeed we can find $\nu\geq 1$ such that
  $$\nu^{-1}\leq H(x)\leq \nu\quad\text{for every $x\in\R^n$ with $|x| = 1$};$$
  this, together with the homogeneity property \emph{(ii)}, readily implies \eqref{eq:HequivalentEu}.\\

 2) follows from \emph{(i)}-\emph{(ii)}, indeed for every $x,y\in\R^n$ we have
  $$\frac{1}{2}H(x+y) = H\Big(\frac{1}{2}x+\frac{1}{2}y\Big) \leq \frac{1}{2}H(x)+\frac{1}{2}H(y),$$
  and this immediately gives the claimed \eqref{eq:triagnleH}.\\
  
3) follows  combining \eqref{eq:HequivalentEu} with \eqref{eq:triagnleH}, indeed we obtain the following estimates, holding true for every $x,y\in\R^n$:
  \vspace*{0.05cm}
  
  \emph{a)}\,\,$H(x)-H(y)\leq H(x-y)\leq \nu|x-y|$;
  
  \emph{b)}\,\,$H(y)-H(x)\leq H(y-x)\leq \nu|y-x| = \nu|x-y|$.
  \vspace*{0.05cm}
  
  \noindent Gathering \emph{a)} and \emph{b)}, we then conclude that
  \begin{equation} \label{eq:HLipschitz}
   |H(x)-H(y)|\leq \nu|x-y| \quad \text{for every $x,y\in\R^n$},
  \end{equation}
  and this proves that $H$ is globally Lipschitz-continuous in $\R^n$.\\

4) follows from \emph{(ii)} and \emph{(iv)}, indeed, given any $x\in\R^n\setminus\{0\}$, we have
  \begin{align*}
   H(x) = \frac{\d}{\d t}\Big|_{t = 1}H(tx) = \langle \nabla H(x),x\rangle,
  \end{align*}
  and this is precisely the claimed \eqref{eq:EulertypeH}.\\
Finally, to prove \eqref{eq:nablaHBd} we first observe that, by 
  combining \eqref{eq:EulertypeH} with property \emph{(iii)}, we de\-rive that 
  $\text{$\nabla H(x)\neq 0$ for every $x\in\R^n\setminus\{0\}$};$
  thus, since $|\nabla H|\in C(\R^n\setminus\{0\})$
  \emph{(}see property \emph{(iv)}\emph{)}, there exists
  a u\-ni\-ver\-sal constant $\vartheta
  \geq 1$ such that
  $$(\star)\qquad\quad
  0 < \vartheta^{-1}\leq |\nabla H(x)|\leq \vartheta\quad\text{for every $x\in\R^n$ with $|x| = 1$}.$$
  On the other hand, since $H$ is positively homogeneous of degree $1$
  \emph{(}see property \emph{(ii)}\emph{)}, a direct computation shows that
  $|\nabla H|$ is positively homogeneous of degree $0$; this,
  together with $(\star)$, readily gives the desired \eqref{eq:nablaHBd}.
  \end{proof}
\end{lem}
\medskip

Moreover, we mention that since $H$ is positive $1$-homogeneous, then $H^2$ is positive $2$-homogeneous, so there exists $\Upsilon>0$ such that
$$\langle \frac{1}{2} \nabla^2H^2(\xi) \eta, \eta\rangle \leq \Upsilon |\eta|^2 \text{ for all } \xi \in \R^n\setminus \{0\}, \ \eta \in \R^n,$$
i.e. the largest eigenvalue of the matrix $\frac{1}{2} \nabla^2H^2(\xi)$ is uniformly bounded from above for $\xi \neq 0$ by some positive constant $\Upsilon > 0$. Furthermore, it is possible to show that $(v)$ is equivalent to prove that there exists $\upsilon>0$ such that
$$\langle \frac{1}{2} \nabla^2H^2(\xi) \eta, \eta\rangle \geq \upsilon |\eta|^2 \text{ for all } \xi \in \R^n\setminus \{0\}, \ \eta \in \R^n,$$
see \cite{AnCiCiFaMa, AnCiFa, CoFaVa2} for further details. Summing up, we have the following uniformly ellipticity condition
\begin{equation}\label{eq:ellipticity} \tag{$v'$}
	\upsilon |\eta|^2 \leq \langle \frac{1}{2} \nabla^2H^2(\xi) \eta, \eta\rangle \leq \Upsilon |\eta|^2 \text{ for all } \xi \in \R^n\setminus \{0\}, \ \eta \in \R^n.
\end{equation}
In what follows, an important role will be played by the {\it dual norm} $H_0: \mathbb{R}^n \to \mathbb{R}$, which is defined as
\begin{equation}\label{eq:DualH0}
H_0 (\zeta):= \sup_{H(\xi)=1}\langle \zeta, \xi \rangle,
\end{equation}
where $\langle \cdot,\cdot \rangle$ denotes the standard Euclidean scalar product. We will also need the map
\begin{equation}\label{eq:hatH0}
\hat{H}_0:\mathbb{R}^n \to \mathbb{R},\qquad
\hat{H}_0 (\zeta):= H_0(-\zeta).
\end{equation}

\begin{rem} \label{rem:propertiesH0} 
It is not difficult to recognize that, since $H$ satisfies properties $(i)$-to-$(iv)$,
both the functions $H_0$ and $\hat{H}_0$ satisfy properties $(i)$-to-$(iii)$;
thus, \emph{all} the properties \emph{1)}-to-\emph{4)} in Lemma
\ref{lem:propHeasy} \emph{(}which do not exploit the regularity of $H$\emph{)} also hold
for $H_0$ and $\hat{H}_0$.
\end{rem}

Concerning the validity of property $(iv)$, namely the smoothness of $H_0$, we have that this is direct consequence of assumption $(v)$ or $(v')$, for which we refer to Corollary 1.7.3 in \cite{Sch}. 
\begin{lem}[Corollary 1.7.3 in \cite{Sch}] Let $H:\mathbb{R}^{n}\to \mathbb{R}$ be such that $(i)$-to-$(iv)$ hold. Then
$$H_0\in C^1(\R^n\setminus\{0\})\,\,
\text{ if and only if }\,\,\text{$B_1^H := \{x\in\R^n:\,H(x)<1\}$ is uniformly convex  }.
$$
Hence, the regularity of the dual norm $H_0$ is equivalent to assumption $(v)$.
\end{lem}
Moreover, thanks to the previous result, we have that $H$ and $H_0$ satisfy the following identities.
\begin{lem}[Lemma 3.1  in \cite{CiSal}]
Assume that $H$ satisfies assumption (i)-(v). Then, the following notable identities hold
\begin{equation}\label{eq:HGradH0}
 H(\nabla H_0(\zeta))\equiv 1,\qquad H_0(\nabla H(\zeta)) \equiv 1,
\end{equation}
for every  $\zeta\neq 0$.
\end{lem}

\noindent\textbf{ii)\,\,Anisotropic Talenti functions.} 
Since we are dealing with critical problems related to the anisotropic $p$-Laplacian $-\Delta_{p}^{H}$, it is crucial to recall the basic results related to what we can call {\em anisotropic Talenti functions}. Recently, in \cite{CFR}, Ciraolo, Figalli and the third named author considered the following problem 
\begin{equation}\label{eq:ProblemCriticalRn}
	\left\{\begin{array}{rl}
		-\Delta^{H}_{p}u = u^{p^{\star}-1}  & \textrm{in   } \R^n,\\
		u >0 & \textrm{in } \R^n,
	\end{array}\right.
\end{equation}
where $u \in D^{1,p}(\R^n)$. There it has been proved that the only energy solutions of \eqref{eq:ProblemCriticalRn} are of the following
form
\begin{equation}\label{eq:HTalentiane}
U^{H}_{\mu,x_0}(x):= \left( \dfrac{\mu^{1/p-1} c_{n,p}}{\mu^{p/p-1}+ \hat{H}_{0}(x-x_0)^{p/p-1}}\right)^{\tfrac{n-p}{p}}, 
\end{equation} 
\noindent for some $\mu>0$, $x_0 \in \mathbb{R}^n$ and
\begin{equation*}
c_{n,p} := n^{1/p} \left(\dfrac{n-p}{p-1}\right)^{\tfrac{p-1}{p}}.
\end{equation*}
We point out that assumption \cite[(1.11)]{CFR} is in force since it is equivalent to assume $(v')$, and hence $(v)$.
The family of functions defined in \eqref{eq:HTalentiane} are also 
\emph{extremals  of the anisotropic Sobolev inequality} (see \cite[Appendix A]{CFR})
\begin{equation} \label{eq:AnisotropicSobolev} 
\mathcal{S}_{H} \|u\|^p_{L^{p^{\star}}(\mathbb{R}^n)} \leq \|H(\nabla u)\|^p_{L^{p}(\mathbb{R}^n)},
\end{equation}
\noindent and coincide with the Aubin-Talenti functions (see, e.g., \cite{Aub, CNV, Tal}) 
when $H(\cdot) = |\cdot|$. 
We point out that $\mathcal{S}_H$ denotes the best Sobolev constant which depends only on $n$ and $p$.
\medskip

For future use, we recall that
\begin{equation*}
\begin{aligned}
\nabla U^{H}_{1,0}(x)&= c_{n,p}^{(n-p)/p} \nabla \left( \left( 1+H_{0}(-x)^{p/p-1}\right)^{(p-n)/p}\right) \\
&=c_{n,p}^{(n-p)/p} \left( \dfrac{n-p}{p-1}\right) \left( 1+H_{0}(-x)^{p/p-1}\right)^{-n/p}H_{0}(-x)^{1/(p-1)}\nabla H_{0}(-x),
\end{aligned}
\end{equation*}
\noindent and hence
\begin{equation}\label{eq:HGradU10}
H(\nabla U^{H}_{1,0}(x))^p = c_{n,p}^{n-p}\left( \dfrac{n-p}{p-1}\right)^p \dfrac{H_0(-x)^{p/(p-1)}}{\left(1+H_{0}(-x)^{p/(p-1)}\right)^n}, 
\end{equation}
\noindent where we used \eqref{eq:HGradH0}.

Finally, let us recall that the Aubin-Talenti bubbles defined in \eqref{eq:HTalentiane}, for every $\mu>0$ satisfy the following
\begin{equation}\label{eq:PhoR^N}
	\|H(\nabla U^{H}_{\mu,x_0})\|_{L^p(\R^n)}^p = \|U^{H}_{\mu,x_0}\|_{L^{p^\star}(\R^n)}^{p^\star}=\mathcal{S}_H^{n/p},
	\end{equation}
where $\mathcal{S}_H$ is the best Sobolev constant in the critical Sobolev embedding $D^{1,p}(\R^n)$ in $L^{p^\star}(\R^n)$. The previous identity is a direct consequence of the  Pohozaev's identity \cite[see Theorem 1.3]{MS}.
\medskip

 Following the strategy in \cite{BN} and \cite{GAP}, and supposing $0\in \Omega$ (w.l.o.g. due to the translation invariance of \eqref{eq:Problem}), we have are able to deduce the following estimate for suitable truncations of the Aubin-Talenti bubbles.

\begin{lem}\label{lem:talentiane}
	Let $\varphi \in C^{\infty}_{0}(\Omega)$ be a cut-off function such that $0\leq \varphi\leq 1$ and $\varphi =1$ in $B_r(0)\subset\subset \Omega$. Let us further define the function
	\begin{equation}\label{eq:etaepsilon}
		\eta_{\varepsilon}(x):= \dfrac{\varphi(x)}{\left(\varepsilon + H_0(-x)^{p/p-1} \right)^{(n-p)/p}}, \quad \textrm{for }\varepsilon >0.
	\end{equation}
	Then, the following hold:
	\begin{equation}\label{eq:Stima_H_nabla_eta}
		\|H(\nabla \eta_{\varepsilon})\|^p_{L^{p}(\Omega)} = \dfrac{1}{c_{n,p}^{n-p} \varepsilon^{(n-p)/p}} \|H(\nabla U^H_{1,0})\|^p_{L^{p}(\mathbb{R}^n)} + O(1), \quad \textrm{as } \varepsilon \to 0^{+}.
	\end{equation}
	\begin{equation}\label{eq:Stima_eta_p_star}
		\int_{\Omega}\eta_{\varepsilon}^{p^{\star}}\, dx = \dfrac{1}{c_{n,p}^n\varepsilon^{n/p}} \|U^{H}_{1,0}\|^{p^{\star}}_{L^{p^{\star}}(\mathbb{R}^n)} + O(1), \quad \textrm{as } \varepsilon \to 0^+.
	\end{equation}
	Moreover, as $\varepsilon \to 0^+$, we have 
	\begin{equation}\label{eq:Stima_eta_p}
		\int_{\Omega}\eta_{\varepsilon}^p \, dx 
		= \left\{ \begin{array}{rl}
			\dfrac{1}{c_{n,p}^{n-p} \varepsilon^{(n-p^2)/p}}\|U^{H}_{1,0}\|^p_{L^{p}(\mathbb{R}^n)} + O(1) & \textrm{if } n>p^2,\\
			\frac{\omega(p-1)}{p}|\log(\e)|+O(1) &\textrm{if } n=p^2,
		\end{array}\right.
	\end{equation}
	and for $p<q<p^\star-1$ we have
	\begin{equation}\label{eq:Stima_eta_q}
		\begin{split}
			\int_{\Omega} \eta_{\varepsilon}^{q}\, dx 
			& \geq \frac{c}{\varepsilon^{\frac{q(n-p)}{p}-\frac{n(p-1)}{p}}},
		\end{split}
	\end{equation}
	where $c$ is a positive constant depending only on $n$ and $p$.
	\end{lem}
	\begin{proof}
		We start computing $\nabla \eta_{\varepsilon}$, getting
		\begin{equation*}
			\begin{aligned}
				\nabla \eta_{\varepsilon}(x) &= \dfrac{\nabla \varphi(x)}{\left(\varepsilon + H_0(-x)^{p/p-1} \right)^{(n-p)/p}} \\
				&\qquad+ \varphi(x) \left(\dfrac{n-p}{p-1}\right)\left(\varepsilon + H_0(-x)^{p/p-1} \right)^{-n/p}H_0(-x)^{1/(p-1)}\nabla H_0(-x).
			\end{aligned}
		\end{equation*}
		By Lagrange Theorem, for $t\in (0,1)$ and setting
		\begin{equation}\label{eq:DefXi}
			\xi = \xi_{x,t}:=  \dfrac{\varphi(x)\left(\dfrac{n-p}{p-1}\right) H_{0}(-x)^{1/(p-1)}\nabla H_{0}(-x)}{\left(\varepsilon + H_0(-x)^{p/p-1} \right)^{n/p}} +  \dfrac{t \,\nabla \varphi(x)}{\left(\varepsilon + H_0(-x)^{p/p-1} \right)^{(n-p)/p}}, 
		\end{equation} 
		we get
		
		\begin{equation}
			\begin{aligned}
				H(\nabla \eta_{\varepsilon})^p&= H \left( \dfrac{\varphi(x)\left(\dfrac{n-p}{p-1}\right) H_{0}(-x)^{1/(p-1)}\nabla H_{0}(-x)}{\left(\varepsilon + H_0(-x)^{p/p-1} \right)^{n/p}}\right)^{p}\\
				&\qquad \qquad + p H(\xi_{x,t})^{p-1} \left\langle \nabla H(\xi), \dfrac{\nabla \varphi(x)}{\left(\varepsilon + H_0(-x)^{p/p-1} \right)^{(n-p)/p}}\right\rangle\\
				&= \varphi(x)^p \left(\dfrac{n-p}{p-1}\right)^p  \dfrac{H_{0}(-x)^{p/p-1}}{\left(\varepsilon + H_0(-x)^{p/p-1} \right)^{n}} \\
				&\qquad \qquad + p H(\xi_{x,t})^{p-1} \left\langle \nabla H(\xi), \dfrac{\nabla \varphi(x)}{\left(\varepsilon + H_0(-x)^{p/p-1} \right)^{(n-p)/p}}\right\rangle
			\end{aligned}
		\end{equation}
		\noindent where we used \eqref{eq:HGradH0}.
		Now, recalling the properties of $\varphi$, we find
		\begin{equation} \label{eq:Anegliappunti}
			\begin{aligned}
				\int_{\Omega}H(\nabla \eta_{\varepsilon})^p \, dx &= \left(\dfrac{n-p}{p-1}\right)^p \int_{\Omega}\dfrac{H_0(-x)^{p/(p-1)}}{\left(\varepsilon + H_0(-x)^{p/(p-1)}\right)^n}\, dx \\
				&\quad - \left(\dfrac{n-p}{p-1}\right)^p \int_{\Omega}\dfrac{(1-\varphi(x)^p)H_0(-x)^{p/(p-1)}}{\left(\varepsilon + H_0(-x)^{p/(p-1)}\right)^n}\, dx \\
				&\quad + p \int_ {\Omega}\dfrac{H(\xi_{x,t})^{p-1}}{\left(\varepsilon + H_0(-x)^{p/(p-1)}\right)^{n-p}}\langle \nabla H(\xi_{x,t}), \nabla \varphi(x)\rangle\, dx\\
				&=: I - II + III. 
			\end{aligned} 
		\end{equation}
		We start estimating $I$.
		\begin{equation*}
			\begin{aligned}
				I&= \left(\dfrac{n-p}{p-1}\right)^p \int_{\mathbb{R}^n}\dfrac{H_0(-x)^{p/(p-1)}}{\left(\varepsilon + H_0(-x)^{p/(p-1)}\right)^n}\, dx \\
				&\quad - \left(\dfrac{n-p}{p-1}\right)^p \int_{\mathbb{R}^{n}\setminus \Omega}\dfrac{H_0(-x)^{p/(p-1)}}{\left(\varepsilon + H_0(-x)^{p/(p-1)}\right)^n}\, dx =: I' - I''.
			\end{aligned}
		\end{equation*}
		Now, we proceed with $I'$: after the change of variable $x=\varepsilon^{(p-1)/p}y$, 
		and recalling \eqref{eq:HGradU10}, we find
		
		\begin{equation*}
			\begin{aligned}
				I'&=  \left(\dfrac{n-p}{p-1}\right)^p \varepsilon^{(p-n)/p} \int_{\mathbb{R}^n}\dfrac{H_0(-y)^{p/(p-1)}}{\left(1 + H_0(-y)^{p/(p-1)}\right)^n}\\
				&=\dfrac{1}{c_{n,p}^{n-p} \varepsilon^{(n-p)/p}} \|H(\nabla U_{1,0})\|^p_{L^{p}(\mathbb{R}^n)}.
			\end{aligned}
		\end{equation*}
		
		Passing to $I''$, keeping in mind that $0\in \Omega$, $n>p$ and that $H_0(\cdot) \sim \|\cdot\|$
		(see Remark \ref{rem:propertiesH0}), we find
		
		\begin{equation}\label{eq:StimaI''}
			|I''| \leq c\int_{\mathbb{R}^n \setminus \Omega}\dfrac{|y|^{p/(p-1)}}{|y|^{np/(p-1)}}\, dy < +\infty.
		\end{equation}
		Therefore, 
		\begin{equation} \label{eq:StimaIdausare}
			I = \dfrac{1}{c_{n,p}^{n-p} \varepsilon^{(n-p)/p}} \|H(\nabla U^{H}_{1,0})\|^p_{L^{p}(\mathbb{R}^n)} + O(1), \quad \textrm{as } \varepsilon \to 0^{+}.
		\end{equation}
		
		Regarding $II$, since $\varphi =1$ on $B_{r}(0)\subset\subset \Omega$, we easily recognize that
		\begin{equation}\label{eq:stimaIIdausare}
		 II = \left(\dfrac{n-p}{p-1}\right)^p \int_{\Omega \setminus B_{r}(0)}\dfrac{(1-\varphi(x)^p)H_0(-x)^{p/(p-1)}}{\left(\varepsilon + H_0(-x)^{p/(p-1)}\right)^n}\, dx = O(1),
		 \end{equation}
		\noindent since we can argue as in \eqref{eq:StimaI''}.
		
		We are now left with $III$. As for $II$, it reduces to an integral over $\Omega \setminus B_{r}(0)$. Now, recalling the definition of $\xi_{x,t}$ in \eqref{eq:DefXi}, we have that there exists a positive constant $C>0$ (depending on $n, p, r, H_0$ and $\Omega$) such that
		\begin{equation*}
			\text{$|\xi_{x,t}|\leq C$ out of $B_r(0)$},
		\end{equation*}
		
		\noindent and hence $H(\xi_{x,t})^{p-1}\leq C.$ Therefore,
		\begin{equation}\label{eq:stimaIIIdausare}
			|III| \leq C \int_{\Omega \setminus B_{r}(0)}\dfrac{1}{|x|^{p(n-p)/(p-1)}}\, dx <+\infty.
		\end{equation}
		Gathering 
		\eqref{eq:StimaIdausare}, \eqref{eq:stimaIIdausare} and 
		\eqref{eq:stimaIIIdausare}, from \eqref{eq:Anegliappunti} we
		obtain \eqref{eq:etaepsilon}.
		\medskip
		
		Arguing in a similar fashion, we have
		
		\begin{equation}
			\begin{aligned}
				\int_{\Omega}\eta_{\varepsilon}^{p^{\star}}\, dx &= \int_{\Omega}\dfrac{1}{\left(\varepsilon + H_0(-x)^{p/(p-1)}\right)^{n}}\, dx - \int_{\Omega}\dfrac{1-\varphi(x)^{p^{\star}}}{\left(\varepsilon + H_0(-x)^{p/(p-1)}\right)^{n}}\, dx\\
				&=: IV + V.
			\end{aligned} 
		\end{equation}
		As before, 
		\begin{equation}
			\begin{aligned} 
				IV &=\int_{\mathbb{R}^n}\dfrac{1}{\left(\varepsilon + H_0(-x)^{p/(p-1)}\right)^{n}}\, dx - \int_{\mathbb{R}^n \setminus \Omega}\dfrac{1}{\left(\varepsilon + H_0(-x)^{p/(p-1)}\right)^{n}}\, dx\\
				&= IV' - IV''.
			\end{aligned}
		\end{equation}  
		Using the change of variable $x = \varepsilon^{(p-1)/p}y$, we get that
		$$IV' = \dfrac{1}{c_{n,p}^n
			\varepsilon^{n/p}} \|U^{H}_{1,0}\|^{p^{\star}}_{L^{p^{\star}}(\mathbb{R}^n)},$$
		\noindent while, similarly as before,
		$$|IV''|, |V| \leq C < +\infty,$$
		\noindent yielding \eqref{eq:Stima_eta_p_star}.
		
		We are left to estimate the $L^p$-norm of $\eta_\varepsilon$. To this end, we distinguish two cases.
		\vspace*{0.1cm}

		a)\,\,$n>p^2$. In this case, we have
		\begin{equation}
			\begin{aligned}
				\int_{\Omega}\eta_{\varepsilon}^p \, dx &= \int_{\Omega} \dfrac{1}{\left(\varepsilon + H_0(-x)^{p/(p-1)}\right)^{n-p}}\, dx -  \int_{\Omega} \dfrac{1-\varphi(x)^p}{\left(\varepsilon + H_0(-x)^{p/(p-1)}\right)^{n-p}}\, dx\\
				&= \int_{\mathbb{R}^n}\dfrac{1}{\left(\varepsilon + H_0(-x)^{p/(p-1)}\right)^{n-p}}\, dx - \int_{\mathbb{R}^n \setminus \Omega}\dfrac{1}{\left(\varepsilon + H_0(-x)^{p/(p-1)}\right)^{n-p}}\, dx \\
				&\quad -  \int_{\Omega \setminus B_{r}(0)} \dfrac{1-\varphi(x)^p}{\left(\varepsilon + H_0(-x)^{p/(p-1)}\right)^{n-p}}\, dx\\
				&=:VI - VII - VIII.
			\end{aligned}
		\end{equation}
		Now, by changing variable $x = \varepsilon^{(p-1)/p}y$, we get
		\begin{equation*}
			VI = \dfrac{1}{c_{n,p}^{n-p} \varepsilon^{(n-p^2)/p}}\|U^{H}_{1,0}\|^p_{L^{p}(\mathbb{R}^n)}.
		\end{equation*}
		For the remaining ones, we argue similarly as before: on the one hand, we have
		\begin{equation*}
			|VIII| \leq C \int_{\Omega \setminus B_{r}(0)}\dfrac{1}{|x|^{(pn-p^2)/(p-1)}}\, dx < +\infty;
		\end{equation*}
		\noindent on the other hand, since we are assuming $n > p^2$, we obtain
		\begin{equation*}
			|VII| \leq C \int_{\mathbb{R}^n \setminus \Omega}\dfrac{1}{|x|^{(pn-p^2)/(p-1)}}\, dx < +\infty.
		\end{equation*}
		All in all, if $n > p^2$,
		\begin{equation}
			\int_{\Omega}\eta_{\varepsilon}^p \, dx = \dfrac{1}{c_{n,p}^{n-p} \varepsilon^{(n-p^2)/p}}\|U^{H}_{1,0}\|^p_{L^{p}(\mathbb{R}^n)} + O(1).
		\end{equation}
		
		b)\,\,$n = p^2$. In this case, by arguing \emph{exactly} as before, we have
		\begin{align*}
			\int_{\Omega}\eta_{\varepsilon}^p \, dx &= 
			\int_{\Omega} \dfrac{1}{\left(\varepsilon + H_0(-x)^{p/(p-1)}\right)^{p^2-p}}\, dx -  \int_{\Omega} \dfrac{1-\varphi(x)^p}{\left(\varepsilon + H_0(-x)^{p/(p-1)}\right)^{p^2-p}}\, dx\\
			& = \int_{\Omega} \dfrac{1}{\left(\varepsilon + \hat{H}_0(x)^{p/(p-1)}\right)^{p^2-p}}\, dx+O(1).
		\end{align*}
		However, in order to estimate the remaining integral we cannot proceed as in the previous case, since
		the function $1/|x|^{(pn-p^2)/(p-1)} = 1/|x|^n$ \emph{is not integrable at infinity}. 
		To overcome this issue, we adapt to our anisotropic context the approach in \cite{BN}.
		\vspace*{0.05cm}
		
		First of all we observe that, since $0\in\Omega$ and $\hat{H}_0$ is globally equivalent to
		the Euclidean norm (see Remark \ref{rem:propertiesH0}), there exist $R_1,R_2 > 0$ such that
		\begin{align*}
			& \int_{\{\hat{H}_0<R_1\}} \dfrac{1}{\left(\varepsilon + \hat{H}_0(x)^{p/(p-1)}\right)^{p^2-p}}\, dx
			\leq\int_{\Omega} \dfrac{1}{\left(\varepsilon + \hat{H}_0(x)^{p/(p-1)}\right)^{p^2-p}}\, dx \\
			& \qquad \qquad \leq \int_{\{\hat{H}_0<R_2\}} 
			\dfrac{1}{\left(\varepsilon + \hat{H}_0(x)^{p/(p-1)}\right)^{p^2-p}}\, dx;
		\end{align*}
		as a consequence, since we have
		\begin{align*}
			& 0\leq \int_{\{\hat{H}_0<R_2\}} \dfrac{1}{\left(\varepsilon + \hat{H}_0(x)^{p/(p-1)}\right)^{p^2-p}}\, dx
			- 
			\int_{\Omega} \dfrac{1}{\left(\varepsilon + \hat{H}_0(x)^{p/(p-1)}\right)^{p^2-p}}\, dx \\
			& \qquad
			\leq
			\int_{\{\hat{H}_0<R_2\}} \dfrac{1}{\left(\varepsilon + \hat{H}_0(x)^{p/(p-1)}\right)^{p^2-p}}\, dx
			-\int_{\{\hat{H}_0<R_1\}} \dfrac{1}{\left(\varepsilon + \hat{H}_0(x)^{p/(p-1)}\right)^{p^2-p}}\, dx \\
			&\qquad
			= \int_{\{R_1\leq \hat{H}_0<R_2\}} \dfrac{1}{\left(\varepsilon + \hat{H}_0(x)^{p/(p-1)}\right)^{p^2-p}}\, dx\\
			& \qquad\leq \int_{\{R_1\leq \hat{H}_0<R_2\}} 
			\dfrac{1}{\hat{H}_0(x)^{p^2}}\, dx = \mathbf{c}(R_1,R_2) < \infty,
		\end{align*}
		we obtain the following identity
		\begin{equation} \label{eq:dovefareCoarea}
			\int_{\Omega} \dfrac{1}{\left(\varepsilon + \hat{H}_0(x)^{p/(p-1)}\right)^{p^2-p}}\, dx
			= \int_{\{\hat{H}_0<R_2\}} \dfrac{1}{\left(\varepsilon + \hat{H}_0(x)^{p/(p-1)}\right)^{p^2-p}}\, dx
			+O(1).
		\end{equation}
		Now, recalling that $\hat{H}_0$ is globally Lipschitz-continuous in $\R^n$
		(see Remark \ref{rem:propertiesH0}), and taking into account
		\eqref{eq:nablaHBd}, we are entitled to apply Federer's Coarea Formula
		to the integral in the right-hand side of \eqref{eq:dovefareCoarea}
		(see, e.g., \cite{EvGar}): this gives
		\begin{align*}
			& \int_{\{\hat{H}_0<R_2\}} \dfrac{1}{\left(\varepsilon + \hat{H}_0(x)^{p/(p-1)}\right)^{p^2-p}}\, dx
			\\
			& \qquad = \int_0^{R_2}\frac{1}{\left(\varepsilon + \rho^{p/(p-1)}\right)^{p^2-p}}
			\Big(\int_{\{\hat{H}_0 = \rho\}}\frac{1}{|\nabla \hat{H}_0|}\,d\mathcal{H}^{p^2-1}\Big)d\rho \\
			&
			\qquad (\text{using the homogeneity of $\hat{H}_0$ and of $|\nabla \hat{H}_0|$}) \\
			& \qquad
			= \omega\,\int_0^{R_2}\frac{\rho^{p^2-1}}{\left(\varepsilon + \rho^{p/(p-1)}\right)^{p^2-p}},
		\end{align*}
		where we have used the fact that $n = p^2$, and
		$$\omega = \int_{\{\hat{H}_0 = 1\}}\frac{1}{|\nabla \hat{H}_0|}\,d\mathcal{H}^{p^2-1}.$$
		To proceed further, we then claim that
		\begin{equation} \label{eq:claimAnalisI}
			\int_0^{R_2}\frac{\rho^{p^2-1}}{\left(\varepsilon + \rho^{p/(p-1)}\right)^{p^2-p}}
			=
			\int_0^{R_2}\frac{\rho^{p^2-1}}{\varepsilon^{p^2-p} + \rho^{p^2}}\,d\rho+O(1).
		\end{equation}
		Taking this claim for granted for a moment, we can
		complete the estimate of $\|\eta_\e\|^p_{L^p(\R^n)}$ in this case: in fact, by combining
		\eqref{eq:dovefareCoarea} with \eqref{eq:claimAnalisI}, we obtain
		\begin{align*}
			& \int_{\Omega} \dfrac{1}{\left(\varepsilon + \hat{H}_0(x)^{p/(p-1)}\right)^{p^2-p}}\, dx
			= \omega\,\int_0^{R_2}\frac{\rho^{p^2-1}}{\varepsilon^{p^2-p} + \rho^{p^2}}\,d\rho+O(1) \\
			& \qquad = \frac{\omega}{p^2}\,\big(\log(\varepsilon^{p^2-p} + R_2^{p^2})-\log(\e^{p^2-p})\big)
			+ O(1)
			\\
			& \qquad = \frac{\omega(p-1)}{p}|\log(\e)|+O(1)\qquad\text{as $\e\to 0^+$}.
		\end{align*}
		Hence, we are left to prove \eqref{eq:claimAnalisI}. To this end it suffices to observe that,
		by performing the change of variables $\rho = s^{1-1/p}$, we have the following identity
		\begin{align*}
			& \int_0^{R_2}{\rho^{p^2-1}}\left|\frac{1}{\left(\varepsilon + \rho^{p/(p-1)}\right)^{p^2-p}}
			-\frac{1}{\varepsilon^{p^2-p} + \rho^{p^2}}\right|d\rho \\
			& \qquad = \int_0^{R_2\e^{1/p-1}}
			{s^{p^2-1}\left|\frac{1}{\left(1 + s^{p/(p-1)}\right)^{p^2-p}}
				-\frac{1}{1 + s^{p^2}}\right|}ds \\
			& \qquad \leq
			\int_0^{\infty}
			\underbrace{s^{p^2-1}\left|\frac{1}{\left(1 + s^{p/(p-1)}\right)^{p^2-p}}
				-\frac{1}{1 + s^{p^2}}\right|}_{= f(s)}ds.
		\end{align*}
		Since $f\in C([0,\infty))$, we clearly have $f\in L^1([0,1])$; moreover,
		\begin{align*}
			\int_1^\infty f(s)\,ds & = \int_1^\infty
			\frac{1}{s}\left|\frac{1}{s^{p^2}}-\frac{p^2-p}{s^{p/(p-1)}}+o\Big(\frac{1}{s^{p^2}}\Big)
			+ o\Big(\frac{1}{s^{p/(p-1)}}\Big)\right|ds \\
			& \leq c\int_1^\infty\Big(\frac{1}{s^{p^2+1}}+\frac{1}{s^{p/(p-1)+1}}\Big)ds < \infty,
		\end{align*}
		and this completes the proof of \eqref{eq:claimAnalisI}.
		
		Finally, we need to prove \eqref{eq:Stima_eta_q}. Since $0 \in \Omega$, let $\tau > 0$ be such that 
		$$\mathcal{O} = B_\tau^{\hat H_0} := \{\hat{H}_0<\tau\} \subset \Omega$$
		(hence, we have $\varphi\equiv 1$ on $\mathcal{O}$), and let 
		$\mu_\e = \e^{1-1/p}$. We then get
		\begin{equation} \label{eq:etalower0}
			\begin{split}
				\int_{\mathbb{R}^n} \eta_{\varepsilon}^{q}\, dx &\geq 
				\int_{\{\hat{H}_0<\tau\}}\frac{1}{(\e+\hat{H}_0(x)^{\frac{p}{p-1}})^{\frac{q(n-p)}{p}}}\,dx \\
				& (\text{setting $x = \e^{\frac{p-1}{p}} z$, and using the homogeneity of $\hat{H}_0$}) \\
				& = \varepsilon^{-\frac{q(n-p)}{p}+\frac{n(p-1)}{p}}
				\int_{\{\hat{H}_0<\tau/\mu_\e\}}\frac{1}{(1+\hat{H}_0(z)^{\frac{p}{p-1}})^{\frac{q(n-p)}{p}}}\,dx \\
				& (\text{by the Coarea Formula}) \\
				& = \,\varepsilon^{-\frac{q(n-p)}{p}+\frac{n(p-1)}{p}}
				\int_0^{\tau/\mu_\e}\frac{\rho^{n-1}}{(1+\rho^{\frac{p}{p-1}})^{\frac{q(n-p)}{p}}}\,d\rho
				\\
				& \geq c\,\varepsilon^{-\frac{q(n-p)}{p}+\frac{n(p-1)}{p}}.
			\end{split}
		\end{equation}
		This ends the proof.
	\end{proof}

\noindent \textbf{iii)\,\,The first eigenvalue of $-\Delta_p^H$}.
Due to its importance in what follows (and to make
the paper as self-contained as possible), 
we explicitly recall here the definition and
the properties of the first variational Dirichlet eigenvalue $\lambda_1^H(\Omega)$ of $-\Delta_{p}^{H}$.
\vspace*{0.1cm}

Taking into account the variational structure of $-\Delta_p^H$, we define
\begin{equation}\label{eq:firstEigenvalue}
\lambda_1^H(\Omega):=\inf_{\substack{v\in W^{1,p}_0(\Omega)\, , \\ v\neq 0}} \dfrac{\int_\Omega H(\nabla v)^p\, dx}{\int_\Omega |v|^p\, dx}\,.
\end{equation}
Then, from \cite[Theorem 3.1]{BFK} we know that the following properties holds.
\begin{itemize}
  \item[1)] $\lambda_1^H(\Omega) > 0$.
  \vspace*{0.1cm}
  
  \item[2)] There exists a \emph{unique, positive function} $u_1\in W^{1,p}_0(\Omega)$
  such that
  $$\|u_1\|_{L^p(\Omega)} = 1\quad
  \text{and}\quad \lambda_1^H(\Omega) = \int_\Omega H(\nabla u_1)^p\, dx;$$
  this function $u_1$ is a weak solution of
\begin{equation}\label{EF}
\begin{cases}
-\Delta^H_ p u=\lambda_1^H(\Omega) u^{p-1}  & \textrm{in   } \Omega,\\
u=0 & \textrm{on } \partial \Omega.
\end{cases}
\end{equation}
and it is called the \emph{principal eigenfunction of $-\Delta^H_p$} (in $\Omega)$.
\vspace*{0.1cm}

\item[3)] $\lambda_1^H(\Omega)$ is \emph{simple} (i.e.,
the vector space of the solutions of \eqref{EF} is one-dimensional).
 \end{itemize}
\medskip

Throughout the sequel, in order to use a compact notation we set
$$\|u\|_{H,p}:=\left(\int_\Omega H(\nabla u)^p \, dx\right)^\frac{1}{p}\qquad (u\in W^{1,p}_0(\Omega).$$
Notice that, since $H$ is globally equivalent to the Eulidean norm
(see Lemma \ref{lem:propHeasy}), we see that $\|\cdot\|_{H,p}$ is a norm on
$W^{1,p}(\Omega)$ 
which is (globally) equivalent to the usual one.

\noindent\textbf{iv)\,\,The functional $\mathcal{J}_{q,\lambda}$}. 
Taking into account all the notions and results recalled so far,
we conclude this section
by
briefly
studying the functional $\mathcal{J}_{q,\lambda}$
defined in \eqref{eq:defFunctionalJ} (which
`encodes' the variational structure of problem \eqref{eq:Problem}).
\medskip

To begin with, we prove that
$\mathcal{J}_{q,\lambda}$ has the mountain-pass geometry.

\begin{lem}[Mountain-pass geometry for $\mathcal{J}_{q,\lambda}$] \label{lem:MPGeometry}
Let $\Omega\subset\R^n$
be a bounded open set with smooth boundary,
and let $1<p<n$. Moreover, let $p\leq q<p^\star$ and
let $\lambda > 0$. 

We assume that either
\begin{equation} \label{eq:assumptions}
 \mathrm{i)}\,\,\text{$q = p$ and $0<\lambda<\lambda_1^H(\Omega)$}\qquad\text{or}\qquad
 \mathrm{ii)}\,\,\textbf{$q > p$}.
\end{equation}
Then, there exist $\rho,\delta_0 > 0$ such that
\begin{equation} \label{eq:MPGeometry}
 \mathcal{J}_{q,\lambda}(u)\geq \delta_0 > 0\quad
 \text{for all $u\in W^{1,p}_0(\Omega)$ with $\|u\|_{H,p} = \rho$}.
\end{equation}
\end{lem}
\begin{proof}
 According to \eqref{eq:assumptions}, we distinguish two cases.
 \medskip
 
 \textsc{Case i):} $q = p$ and $0<\lambda<\lambda_1^H(\Omega)$.
 In this case we first notice that, owing to the ve\-ry definition of
 $\lambda_1^H(\Omega)$ given in \eqref{eq:firstEigenvalue}, 
 for every $u\in W^{1,p}_0(\Omega)$ we have
 \begin{equation} \label{eq:estimpLinearCasei}
  \int_{\Omega}(u^+)^p \, dx 
  \leq \int_{\Omega}|u|^p \, dx 
  \leq \frac{1}{\lambda_1^H(\Omega)}\int_{\Omega}H(\nabla u)^p \, dx
  = \frac{1}{\lambda_1^H(\Omega)}\|u\|_{H,p}^p.
 \end{equation}
 Moreover, by the anisotropic Sobolev inequality \eqref{eq:AnisotropicSobolev}, we also have
 \begin{equation} \label{eq:estimCriticalCaseiii}
    \int_{\Omega}(u^+)^{p^\star}\,dx
    \leq  \int_{\Omega}|u|^{p^\star}\,dx
    \leq \mathcal{S}_H^{-p^\star/p}\|u\|_{H,p}^{p^\star}\quad\forall\,\,u\in W^{1,p}_0(\Omega).
 \end{equation}
 Gathering \eqref{eq:estimpLinearCasei}\,-\,\eqref{eq:estimCriticalCaseiii}
 (and since we are assuming $q = p$), we then get
 \begin{equation} \label{eq:estimtoConclude}
 \begin{split}
  \mathcal{J}_{p,\lambda}(u)
  & \geq \frac{1}{p}\|u\|_{H,p}^p
  -\frac{\lambda}{p\lambda_1^H(\Omega)}\|u\|_{H,p}^p-
  \frac{\mathcal{S}_H^{-p^\star/p}}{p^\star}\|u\|_{H,p}^{p^\star} \\
  & = \|u\|_{H,p}^p\bigg\{\frac{1}{p}\Big(1-\frac{\lambda}{\lambda_1^H(\Omega)}\Big)-
  \frac{\mathcal{S}_H^{-p^\star/p}}{p^\star}\|u\|_{H,p}^{p^\star-p}\bigg\}\quad\forall\,\,u\in W^{1,p}_0(\Omega).
  \end{split}
 \end{equation}
 With estimate \eqref{eq:estimtoConclude} at hand,
 we can easily complete the proof
 of \eqref{eq:MPGeometry} in this case.
 Indeed, since we are assuming $0<\lambda<\lambda_1^H(\Omega)$, we clearly have
 $$\theta = \frac{1}{p}\Big(1-\frac{\lambda}{\lambda_1^H(\Omega)}\Big) > 0;$$
 as a consequence, since $p^\star-p > 0$, we can choose $\rho > 0$ so small that
 $$\frac{1}{p}\Big(1-\frac{\lambda}{\lambda_1^H(\Omega)}\Big)-\frac{\mathcal{S}_H^{-p^\star/p}}{p^\star}\rho^{p^\star-p}  \geq   \frac{\theta}{2}.$$
 This, together with \eqref{eq:estimtoConclude}, thus ensures that 
 \begin{align*}
  \mathcal{J}_{p,\lambda}(u)
  & \geq \|u\|_{H,p}^p\bigg\{\frac{1}{p}\Big(1-\frac{\lambda}{\lambda_1^H(\Omega)}\Big)-
  \frac{\mathcal{S}_H^{-p^\star/p}}{p^\star}\|u\|_{H,p}^{p^\star-p}\bigg\} \\
  & = \rho^p\bigg\{\frac{1}{p}\Big(1-\frac{\lambda}{\lambda_1^H(\Omega)}\Big)-
  \frac{\mathcal{S}_H^{-p^\star/p}}{p^\star}\rho^{p^\star-p}\bigg\}
  \geq \frac{\theta\rho^{p}}{2} = \delta_0 > 0,
 \end{align*}
 provided that $u\in W_0^{1,p}(\Omega)$ and $\|u\|_{H,p} = \rho$. 
 \medskip
 
  \textsc{Case ii):} $p<q<p^\star$.
  In this case we first notice that, by combining H\"older's inequality
  with the anisotropic Sobolev inequality \eqref{eq:AnisotropicSobolev},
  for every $u\in W^{1,p}_0(\Omega)$ we have
  \begin{equation} \label{eq:estimqSuperLinCaseii}
  \begin{split}
   \int_{\Omega}(u^+)^q \, dx 
   & \leq \int_{\Omega}|u|^q \, dx 
   \leq |\Omega|^{1-q/p^\star}\|u\|_{L^{p^\star}(\Omega)}^q 
   \leq c\|u\|_{H,p}^q,
   \end{split}
  \end{equation}
  where $c > 0$ is a suitable constant only depending on $\Omega$.
  Thus, by combining \eqref{eq:estimqSuperLinCaseii} 
  with the above
  \eqref{eq:estimCriticalCaseiii} (which is independent of $q$ and $\lambda$),
  we obtain
  \begin{equation} \label{eq:estimtoConcludeCaseii}
 \begin{split}
  \mathcal{J}_{q,\lambda}(u)
  & \geq \frac{1}{p}\|u\|_{H,p}^p
  -\frac{c\lambda}{q}\|u\|_{H,p}^q-
  \frac{\mathcal{S}_H^{-p^\star/p}}{p^\star}\|u\|_{H,p}^{p^\star} \\
  & (\text{since we are assuming $q > p$}) \\
  & = \|u\|_{H,p}^p\bigg\{\frac{1}{p}-\frac{c\lambda}{q}\|u\|_{H,p}^{q-p}-
  \frac{\mathcal{S}_H^{-p^\star/p}}{p^\star}\|u\|_{H,p}^{p^\star-p}\bigg\}\quad\forall\,\,u\in W_0^{1,p}(\Omega).
  \end{split}
 \end{equation}
 With estimate \eqref{eq:estimtoConcludeCaseii} at hand,
 we can easily complete the proof
 of \eqref{eq:MPGeometry} in this case.
 Indeed, since $p^\star-p > 0$ and $q-p > 0$ (as $p < q$), we
 can choose $\rho>0$ so small that
 $$\frac{1}{p}-\frac{c\lambda}{q}\rho^{q-p}-
  \frac{\mathcal{S}_H^{-p^\star/p}}{p^\star}\rho^{p^\star-p}\geq \frac{1}{2p};$$
 this, together with \eqref{eq:estimtoConcludeCaseii}, thus ensures that
 \begin{align*}
  \mathcal{J}_{q,\lambda}(u)
  & \geq \|u\|_{H,p}^p\bigg\{\frac{1}{p}-\frac{c\lambda}{q}\|u\|_{H,p}^{q-p}-
  \frac{\mathcal{S}_H^{-p^\star/p}}{p^\star}\|u\|_{H,p}^{p^\star-p}\bigg\} \\
  & = 	\rho^p\bigg\{\frac{1}{p}-\frac{c\lambda}{q}\rho^{q-p}-
  \frac{\mathcal{S}_H^{-p^\star/p}}{p^\star}\rho^{p^\star-p}\bigg\}
  \geq \frac{\rho^{p}}{2p} = \delta_0 > 0,
 \end{align*}
 provided that $u\in W_0^{1,p}(\Omega)$ and $\|u\|_{H,p} = \rho$.
\end{proof}
\begin{rem} \label{rem:StrongerMPG}
 By scrutinizing the proof of Lemma \ref{lem:MPGeometry}, it is easy
 to recognize that condition \eqref{eq:MPGeometry} actually holds in the following
 \emph{stronger form}: there exist $\rho,\eta > 0$ such that
 \begin{equation} \label{eq:StrongerMPG}
  \mathcal{J}_{q,\lambda}(u)\geq \eta\|u\|_{H,p}^p\quad\text{for every $u\in W^{1,p}_0(\Omega)$
 with $\|u\|_{H,p}\leq \rho$}.
 \end{equation}
 Clearly, condition \eqref{eq:MPGeometry} follows from \eqref{eq:StrongerMPG} by choosing
 $\delta_0 = \eta\rho^p$.
\end{rem}
\begin{rem} \label{rem:Equivalentnorms}
 We explicitly observe that, since the norm
 $\|\cdot\|_{H,p}$ is equivalent to the usual Sobolev norm, we see that
 condition \eqref{eq:StrongerMPG} \emph{(}and hence condition \eqref{eq:MPGeometry}\emph{)} 
 holds if we replace the norm $\|\cdot\|_{H,p}$ with
 $$\|u\|_{W_0^{1,p}(\Omega)} = \Big(\int_\Omega|\nabla u|^p\,dx\Big)^{1/p}$$ 
 (up to possibly modifying $\rho$).
\end{rem}
We then prove the following technical lemma, in the same spirit of \cite{AG} and \cite{GAP}.

\begin{lem}\label{lem:PSseq}
Assume that  there exists a 
\emph{Palais-Smale sequence $\{u_t\}_{t \geq 0} \subset W^{1,p}_0(\Omega)$ for
the functional $\mathcal{J}_{q,\lambda}$ at level $\alpha \in (0, \mathcal{S}_H^{n/p}/n)$}, that is, 
\begin{itemize}
 \item[i)] $\mathcal{J}_{q,\lambda}(u_t)\to\alpha$ as $t\to+\infty$;
 \item[ii)] $\mathcal{J}_{q,\lambda}'(u_t)\to 0$ as $t\to+\infty$ in $(W_0^{1,p}(\Omega))'$.
\end{itemize}
Then there exists $u \in W_0^{1,p}(\Omega) \setminus \{0\}$ such that $u_t \rightharpoonup  u$ up	to a subsequence, and 
$$\text{$J'_{q,\lambda}(u)[v] = 0$ for all $v \in W^{1,p}_0(\Omega)$}.$$
In particular, $u$ is a solution of problem \eqref{eq:Problem}.
\end{lem}

\begin{proof} 
	We first observe that, since $\{u_t\}_{t\geq 0}$ is a Palais-Smale
	sequence for $\mathcal{J}_{q,\lambda}$ (hence,
	properties i)-ii) holds), we can find a constant $C > 0$ such that
	\begin{align*}
	\mathrm{a)}\,\,&|\mathcal{J}_{q,\lambda}(u_t)| \leq C \\
	\mathrm{b)}\,\,&|\mathcal{J}_{q,\lambda}'(u_t)[u_t]| \leq o(1) \left(\int_\Omega H(\nabla u_t)^p \,dx\right)^\frac{1}{p}= o(1) \cdot \|u_t\|_{H,p}.
	\end{align*}
	We then proceed by steps.
	\medskip
	
	\noindent  \textsc{Step i):}  we firstly prove that $\{u_t\}_{t\geq 0}$  is bounded in $W_0^{1,p}(\Omega)$; hence, we compute 
	\[
	q \mathcal{J}_{q,\lambda}(u_t) - \mathcal{J}_{q,\lambda}'(u_t)[u_t] \leq C + o(1) 
	\cdot \|u_t\|_{H,p},
	\]
	but we also have
	\begin{align*}
		q\mathcal{J}_{q,\lambda}(u_t)- \mathcal{J}'_{q,\lambda}(u_t)[u_t] & = \left(\frac qp-1\right)\int_\Omega H(\nabla u_t)^p\, dx - \lambda \int_\Omega (u^+_t)^q \, dx -\frac{q}{p^\star} \int_\Omega (u_t^+)^{p^\star}\, dx \\
		& \qquad + \lambda \int_\Omega (u^+_t)^q \, dx + \int_\Omega (u_t^+)^{p^\star}\, dx \\
		&\geq\left(\frac{q-p}{p}\right)\int_\Omega H(\nabla u_t)^p\, dx + \left(1-\frac{q}{p^\star}\right) \int_\Omega (u_t^+)^{p^\star}\, dx\\
		&\geq\left(\frac{q-p}{p}\right)\int_\Omega H(\nabla u_t)^p\, dx = \left(\frac{q-p}{p}\right) \|u_t\|_{H,p}^p,
	\end{align*}
	where we used the fact that $1-q/p^\star \geq 0$.  Thus we have that 
	\[
	\|u_t\|_{H,p}^p \leq C + o(1) \|u_t\|_{H,p}.
	\]
	It follows that $\{u_t\}_{t\geq 0}$ is bounded in $W^{1,p}_0(\Omega)$. Hence, there exists $u\in W^{1,p}_0(\Omega)$ such that, up to a subsequence, $u_t\rightharpoonup u$ and therefore $u_t\rightarrow u$ in $L^s(\Omega)$ for all $s<p^\star$. 
	
	In particular, for any $v\in W^{1,p}_0(\Omega)$ we obtain that 
	\begin{align*}
		\mathcal{J}'_{q,\lambda}(u_t)[v]\rightarrow \mathcal{J}'_{q,\lambda}(u)[v]=0.
	\end{align*}
	
	\noindent \textsc{Step ii):} we want to prove that $u \not \equiv 0$. Assume by contradiction that $u\equiv 0$, then, as pointed out above, $u_t\rightharpoonup 0$ in $W^{1,p}_{0}(\Omega)$ implies 
	$$
	u_t \rightarrow 0 \quad \text{in } L^s(\Omega)\, \text{ for all } s<p^\star.
	$$
	Since $\{u_t\}_{t\geq 0}$ is a Palais-Smale 
	sequence (hence, properties i)-ii) hold), 
	from the one hand we have $\mathcal{J}'_{q,\lambda}(u_t)[u_t]=o(1)$;  on the other hand, we also
	have
	\begin{align*}
		\mathcal{J}'_{q,\lambda}(u_t)[u_t]&= \int_\Omega H(\nabla u_t)^p \, dx - \lambda \int_\Omega (u_t^+)^q\, dx - \int_\Omega (u_t^+)^{p^\star} \, dx\\
		&=\int_\Omega H(\nabla u_t)^p \, dx - o(1) - \int_\Omega (u_t^+)^{p^\star} \, dx \,, 
	\end{align*}
	i.e. 
	\begin{equation}\label{AG1}
		\int_\Omega H(\nabla u_t)^p \, dx - \int_\Omega (u_t^+)^{p^\star} \, dx= o(1)\, .
	\end{equation}
	Moreover, from the anisotropic Sobolev inequality we have 
	\begin{equation}\label{AG1bis}
	o(1)\geq \int_\Omega H(\nabla u_t)^p \, dx \left[ 1-\mathcal{S}_H^{-p^\star/p}\left( \int_\Omega H(\nabla u_t)^{p} \, dx\right)^{\frac{p^\star-p}{p}}\right]\, . 
	\end{equation}
	We recall that, since $\{u_t\}_{t\geq 0}$ is bounded in $W^{1,p}_0(\Omega)$, there exists $C>0$ such that
	$$\left(\int_\Omega H(\nabla u_t)^p\, dx \right)^{1/p} \leq C,$$
	for each $t>0$. Hence we can pass to the limit for $t \rightarrow 0^+$, and we deduce that up to subsequences
	\begin{equation}\label{eq:GradLimitL}
		\int_\Omega H(\nabla u_t)^p\, dx \rightarrow L,
	\end{equation}
	where $L \geq 0$. If $L=0$, then by \eqref{AG1} we would have also that 
	$$\int_\Omega (u_t^+)^{p^\star} \, dx \rightarrow 0,$$
	which immediately implies $\mathcal{J}_{q,\lambda}(u_t) \rightarrow 0$ in $\R$, which
	is in contradiction with property i)
	of the sequence $\{u_t\}_{t\geq 0}$.
	Hence,
	\begin{equation}\label{eq:GradLimitnot0}
		\int_\Omega H(\nabla u_t)^p\, dx \not \rightarrow 0,
	\end{equation}
	and thus $L > 0$. This, together with \eqref{AG1bis}, easily implies
	that $L\geq \mathcal{S}_H^{n/p}$, so that
	\begin{equation}\label{AG2}
		\int_\Omega H(\nabla u_t)^p\, dx \geq \mathcal{S}_H^{n/p} + o(1).
	\end{equation}
	
	Finally, from \eqref{AG2} and \eqref{AG1} we have 
	\begin{align*}
		\mathcal{J}_{q,\lambda}(u_t)&= \frac{1}{n} \int_\Omega H(\nabla u_t)^p\, dx + \frac{1}{p^\star}  \left\{ \int_\Omega H(\nabla u_t)^p\, dx - \int_\Omega (u_t^+)^{p^\star}\, dx \right\} + o(1) \\
		&=  \frac{1}{n} \int_\Omega H(\nabla u_t)^p\, dx  + o(1) \\
		&\geq \dfrac{\mathcal{S}_H^{n/p}}{n}+o(1)\, , 
	\end{align*}
	which contradicts the assumption $\alpha< n^{-1}\mathcal{S}_H^{n/p}$.
	This closes the proof.
\end{proof}

\section{Proofs of Theorems \ref{thm:BN-style}, \ref{thm_AG}, \ref{thm:PertNonlinear}}\label{sec:Proof_BN}
Taking into account all
the preliminary results
established in Section \ref{sec:Prel},
we are now ready to prove
all the existence results
stated in the Introduction. 
\medskip

To begin with, we consider the case
of $p$-linear perturbation $q=p$, and we start by dealing with the 
high-dimensional case $1<p^2\leq n$ (hence, we aim at proving
Theorem \ref{thm:BN-style}). 
In order to do this, let us  consider the function, for every fixed $\e > 0$, 
\begin{align*}
	v_\e & = \big(\e^{1/p}c_{n,p}\big)^{\frac{n-p}{p}}\eta_\e = 
	\big(\e^{1/p}c_{n,p}\big)^{\frac{n-p}{p}}\,
	\frac{\varphi(x)}{(\e+\hat{H}_0(x)^{\frac{p}{p-1}})^{\frac{n-p}{p}}}
	= \varphi(x)U^H_{\varepsilon^{p/(p-1)},0}(x)
\end{align*}
where we have set (here and throughout) 
$$\hat{H}_0(x)= H_0(-x)$$ 
and $\eta_{\varepsilon}$ was defined in Lemma \ref{lem:talentiane}.

Now, we would like to construct a path with the right energy, in order to apply the mountain pass theorem. Hence, it holds the following:

\begin{lem}\label{lem:Path_q=p}
	Let $1<p^2\leq n$. Then there exists $\varepsilon>0$ small enough such that 
		\begin{equation*}
			\sup_{t \geq 0}\mathcal{J}_{p,\lambda}(t v_{\varepsilon}) < \dfrac{\mathcal{S}_H^{n/p}}{n}\quad
			\text{for all $\lambda > 0$}.
		\end{equation*}
\end{lem}

\begin{proof}
	The argument of this proof is inspired by \cite{AG, BN, GAP}. 
	Using \eqref{eq:PhoR^N} in  \eqref{eq:Stima_H_nabla_eta} and \eqref{eq:Stima_eta_p_star}, we deduce
	\begin{equation}\label{eq:Stima_H_nabla_v}
		\|H(\nabla v_{\varepsilon})\|^p_{L^{p}(\Omega)} = \mathcal{S}_H^{n/p} + O(\varepsilon^{(n-p)/p}), \quad \textrm{as } \varepsilon \to 0^{+}.
	\end{equation}
	\begin{equation}\label{eq:Stima_v_p_star}
		\|v_{\varepsilon}\|^{p^{\star}}_{L^{p^\star}(\Omega)} = \mathcal{S}^{n/p}_H + O(\varepsilon^{n/p}), \quad \textrm{as } \varepsilon \to 0^+.
	\end{equation}
	Our aim is to show that for $\varepsilon > 0$ small enough it holds the following
	\begin{equation}\label{eq:PSlevel}
		\sup_{t\geq 0} \mathcal{J}_{p,\lambda}(tv_\varepsilon) < \frac{\mathcal{S}_H^{n/p}}{n},
	\end{equation}	
	Let us assume by contradiction that for all $\varepsilon > 0$ there exists $t_\varepsilon>0$ such that
	\begin{equation}\label{eq:contradiction}
		\mathcal{J}_{p,\lambda}(t_\varepsilon v_\varepsilon) \geq \frac{\mathcal{S}_H^{\frac{n}{p}}}{n}.
	\end{equation}
	We note that, up to subsequences, $\{t_\varepsilon\}$ converges since it is bounded from above and below by two positive constants. Since
	\[
	\mathcal{J}_{p,\lambda}(t_\varepsilon v_{\varepsilon}) = \frac{t_\varepsilon^p}{p} \int_{\Omega}H(\nabla v_\varepsilon)^p \, dx - \lambda \frac{ t_\varepsilon^p}{p} \int_{\Omega}v_\varepsilon^p \, dx -  \frac{t_\varepsilon^{p^\star}}{p^\star} \int_{\Omega}v_\varepsilon^{p^\star} \, dx \, ,
	\]
	it is easy to see that $\mathcal{J}_{p,\lambda}(t_\varepsilon v_{\varepsilon}) \rightarrow 0$ if $t_\varepsilon \rightarrow 0$, and $\mathcal{J}_{p,\lambda}(t_\varepsilon v_{\varepsilon}) \rightarrow -\infty$ if $t_\varepsilon \rightarrow +\infty$. In both cases, we get a contradiction with \eqref{eq:contradiction}.
	From \eqref{eq:Stima_H_nabla_v} and \eqref{eq:Stima_v_p_star} we deduce that
	\begin{equation}\label{eq:Stima_H_nabla_vbis}
		\frac{\|H(\nabla (t_\varepsilon v_{\varepsilon}))\|^p_{L^{p}(\Omega)}}{p} \leq \frac{\mathcal{S}_H^{n/p}}{p} + \frac{t_\varepsilon^p-1}{p} \mathcal{S}_H^{n/p} + C_1\varepsilon^{(n-p)/p}, \quad \textrm{as } \varepsilon \to 0^{+},
	\end{equation}
	\begin{equation}\label{eq:Stima_v_p_starnis}
		\frac{\|t_\varepsilon v_{\varepsilon}\|^{p^{\star}}_{L^{p^\star}(\Omega)}}{p^\star} \geq \frac{\mathcal{S}^{n/p}_H}{p^\star} + \frac{t_\varepsilon^{p^\star}-1}{p^\star} \mathcal{S}_H^{n/p} - C_2\varepsilon^{n/p}, \quad \textrm{as } \varepsilon \to 0^+,
	\end{equation}
	where $C_1,C_2$ are suitable positive constant. Hence, subtracting both the left hand side we get
	\begin{equation}\label{eq:intermediate}
	\begin{split}
		&\frac{\|H(\nabla (t_\varepsilon v_{\varepsilon}))\|^p_{L^{p}(\Omega)}}{p} - \frac{\|t_\varepsilon v_{\varepsilon}\|^{p^{\star}}_{L^{p^\star}(\Omega)}}{p^\star} \\
		& \qquad \leq \frac{\mathcal{S}_H^{n/p}}{n} + \left[t_\varepsilon^p-1-\frac{n-p}{n} (t_\varepsilon^{p^\star}-1)\right] \frac{\mathcal{S}_H^{n/p}}{p} + C \varepsilon^{(n-p)/p},
		\end{split}
	\end{equation}
	where $C:=C_1+C_2\varepsilon$. Noticing that it holds the following elementary inequality
	\begin{equation}\label{eq:elementary}
		s^p-1-\frac{n-p}{n}(s^{p^\star}-1) \leq 0 \ \  \text{ for all } s \geq 0,
	\end{equation}
	from \eqref{eq:intermediate} we deduce that
	\begin{equation}\label{eq:estimateenergy1}
		\frac{\|H(\nabla (t_\varepsilon v_{\varepsilon}))\|^p_{L^{p}(\Omega)}}{p} - \frac{\|t_\varepsilon v_{\varepsilon}\|^{p^{\star}}_{L^{p^\star}(\Omega)}}{p^\star} \leq \frac{\mathcal{S}_H^{n/p}}{n}  + C \varepsilon^{(n-p)/p}.
	\end{equation}
	Thanks to \eqref{eq:Stima_eta_p}, we get
	\begin{equation}\label{eq:Stima_v_p}
	\|t_\varepsilon v_\varepsilon\|_{L^p(\Omega)}^p  
	= \left\{ \begin{array}{rl}
		t_\varepsilon^p \varepsilon^{p-1}\|U^{H}_{1,0}\|^p_{L^{p}(\mathbb{R}^n)} + O(\varepsilon^{(n-p)/p}) & \textrm{if } n>p^2,\\[0.15cm]
		c\,t_\varepsilon^p \varepsilon^{p-1} |\log(\e)|+O(\varepsilon^{p-1}) &\textrm{if } n=p^2.
	\end{array}\right.
\end{equation}
	(for some $c > 0$).
	Thus we obtain
	\begin{equation}\label{eq:Stima_v_pfinal}
		\|t_\varepsilon v_\varepsilon\|_{L^p(\Omega)}^p  
		= \left\{ \begin{array}{rl}
			 O(\varepsilon^{p-1}) & \textrm{if } n>p^2,\\
			O(\varepsilon^{p-1}|\log(\e)|) &\textrm{if } n=p^2.
		\end{array}\right.
	\end{equation}
	Combining \eqref{eq:estimateenergy1} and \eqref{eq:Stima_v_pfinal}, there exist 
	$K_1,K_2>0$ such that for $\varepsilon> 0$ sufficiently small we get
	\begin{equation}
		\begin{split}
		\mathcal{J}_{p,\lambda}(t_\varepsilon v_\varepsilon) &= 	\frac{\|H(\nabla (t_\varepsilon v_{\varepsilon}))\|^p_{L^{p}(\Omega)}}{p}  - \frac{\|t_\varepsilon v_{\varepsilon}\|^{p^{\star}}_{L^{p^\star}(\Omega)}}{p^\star} - \lambda \frac{\|t_\varepsilon v_\varepsilon\|_{L^p(\Omega)}^p}{p}\\
		& \leq \left\{ \begin{array}{rl}
		\frac{\mathcal{S}_H^{n/p}}{n}  + C \varepsilon^{(n-p)/p} - K_1 \varepsilon^{p-1} & \textrm{if } n>p^2\\
		\frac{\mathcal{S}_H^p}{p^2}   + C \varepsilon^{p-1} - K_2\varepsilon^{p-1}|\log(\e)| &\textrm{if } n=p^2
		\end{array}\right.\\
		&< 	\frac{\mathcal{S}_H^{n/p}}{n},
		\end{split}
	\end{equation}
	which is in contradiction with \eqref{eq:contradiction}. 
	\end{proof}

	\begin{proof}[Proof of Theorem \ref{thm:BN-style}] 
	 Let $\e > 0$ so small that Lemma \ref{lem:Path_q=p}
	 applies, and let $\bar{t} > 0$ be such that
	 $\bar{t}\,\|v_\e\|_{H,p} > \rho$ and $\mathcal{J}_{p,\lambda}(\bar{t}v_\e) < 0$,
	 where $\rho > 0$ is as in Lemma \ref{lem:MPGeometry}. We then set
	 $$\alpha = \inf_{\gamma\in \mathcal{C}}\sup_{t\in[0,1]}\mathcal{J}_{p,\lambda}(\gamma(t)),
	 \quad
	 \text{$\mathcal{C} = \{\gamma\in C([0,1];W_0^{1,p}(\Omega)):\,\gamma(0) =0,\,\gamma(1) = \bar{t}v_\e\}$}.
	 $$
	 Since $\mathcal{J}_{p,\lambda}$ satisfies the moun\-tain\--pass geometry (see Lemma \ref{lem:MPGeometry}),
	 by the well-known mountain pass theorem we know that $\alpha \geq \delta_0 > 0$, and
	 that there exists a Palais-Smale sequence for $\mathcal{J}_{p,\lambda}$ at level $\alpha$;
	 most importantly, from Lemma \ref{lem:Path_q=p} we derive that
	 $$\alpha \leq \max_{t\geq 0}\mathcal{J}_{p,\lambda}(tv_\e) < \frac{\mathcal{S}_H^{n/p}}{n}.$$
	 We are then entitled to apply Lemma \ref{lem:PSseq}, ensuring that there exists
	 a weak solution of problem \eqref{eq:Problem} in this case (namely, $n \geq p^2$
	 and $q = p$).
\end{proof}
Keeping to
considering the $p$-linear case $q = p$, we now turn to consider the
low-dimensional case $1<p<n<p^2$ (namely,
we aim at proving Theorem \ref{thm_AG}). 
In the Euclidean framework this type of result can be found in \cite{PS} and \cite{GV}. 

Inspired by \cite{AG}, we are now ready to prove the following:

\begin{lem}\label{lem:PathLow}
	Let $1<p<n<p^2$. Then there exists $\varepsilon>0$ small enough such that 
	\begin{equation*}
		\sup_{t \geq 0}\mathcal{J}_{p,\lambda}(t v_{\varepsilon}) < \dfrac{1}{n}(\mathcal{S}_H)^{n/p}\quad
		\text{for all $\lambda\in(\lambda_1^H(\Omega)-\Lambda,\lambda_1^H(\Omega))$},
	\end{equation*}
	where $\Lambda=\mathcal{S}_{H} \, |\Omega|^{-p/n}$.
\end{lem}

\begin{proof}
\noindent  We construct an explicit path whose energy level (w.r.t. $\mathcal{J}_{p,\lambda}$) belongs to the interval  $(0,n^{-1} \mathcal{S}_H^{n/p})$.  Let $u_1$ be the first positive eigenfunction of $-\Delta_p^H$ in $\Omega$.
For $t>0$ we compute 
\begin{align*}
\mathcal{J}_{p,\lambda}(tu_1)=&\frac{t^p}{p} \int_\Omega H(\nabla u_1)^p\, dx - \frac{\lambda t^p}{p} \int_{\Omega}u_1^p \, dx -  \frac{t^{p^\star}}{p^\star} \int_{\Omega}u_1^{p^\star} \, dx \\ 
=& \frac{\lambda_1^H(\Omega)-\lambda}{p} \int_{\Omega}(tu_1)^p \, dx -  \frac{1}{p^\star} \int_{\Omega}(tu_1)^{p^\star} \, dx\, ,
\end{align*}
from this, since $\lambda_1^H(\Omega)-\lambda>0$, by H\"older inequality we get 
$$
\mathcal{J}_{p,\lambda}(tu_1)\leq \frac{\lambda_1^H(\Omega)-\lambda}{p} \left(\int_{\Omega}(tu_1)^{p^\star} \, dx \right)^{p/{p^\star}} |\Omega|^{p/n}-  \frac{1}{p^\star} \int_{\Omega}(tu_1)^{p^\star} \, dx \, ,
$$
i.e. 
$$
\mathcal{J}_{p,\lambda}(tu_1)\leq \frac{\left[\lambda_1^H(\Omega)-\lambda\right]^{n/p}|\Omega|}{n} \, ;
$$
where we used the following (trivial) fact 
$$
\max_{x\geq 0} \left( ax-bx^{n/(n-p)}\right)=\frac{ap}{n}\left[\frac{a(n-p)}{nb}\right]^{(n-p/p)} \quad \text{ for all } a,b>0\, ,
$$
with 
$$
a= \frac{\left[\lambda_1^H(\Omega)-\lambda\right]|\Omega|^{p/n}}{p} \, , \quad b= \frac{1}{p^\star} \, , \quad \text{ and } \quad x=\left(\int_{\Omega}u^{p^\star} \, dx \right)^{p/{p^\star}}\, .
$$
Since $\lambda\in(\lambda_1^H(\Omega)-\Lambda,\lambda_1^H(\Omega))$ then 
$$
\sup_{t\geq 0}\mathcal{J}_{p,\lambda}(t u_1)<\frac{\Lambda^{n/p}|\Omega|}{n}= \dfrac{\mathcal{S}_H^{n/p}}{n}\, . 
$$
This ends the proof.
\end{proof}

\begin{proof}[Proof of Theorem \ref{thm_AG}]
Let $\e > 0$ so small that Lemma \ref{lem:PathLow}
applies, and let $\bar{t} > 0$ be such that
	 $\bar{t}\,\|v_\e\|_{H,p} > \rho$ and $\mathcal{J}_{p,\lambda}(\bar{t}v_\e) < 0$,
	 where $\rho > 0$ is as in Lemma \ref{lem:MPGeometry}. We then set
	 \begin{equation*}
	 \begin{gathered}
	 \alpha = \inf_{\gamma\in \mathcal{C}}\sup_{t\in[0,1]}\mathcal{J}_{p,\lambda}(\gamma(t)),
	 \quad\text{where $\mathcal{C} = \{\gamma\in C([0,1];W_0^{1,p}(\Omega)):\,\gamma(0) =0,\,
	 \gamma(1) = \bar{t}v_\e\}$}.
	 \end{gathered}
	 \end{equation*}
	 Since $\mathcal{J}_{p,\lambda}$ satisfies the moun\-tain\--pass geometry (see Lemma \ref{lem:MPGeometry}),
	 by the well-known mountain pass theorem we know that $\alpha \geq \delta_0 > 0$, and
	 that there exists a Palais-Smale sequence for $\mathcal{J}_{p,\lambda}$ at level $\alpha$;
	 most importantly, from Lemma \ref{lem:PathLow} we derive that
	 $$\alpha \leq \max_{t\geq 0}\mathcal{J}_{p,\lambda}(tv_\e)< \frac{\mathcal{S}_H^{n/p}}{n}.$$
	 We are then entitled to apply Lemma \ref{lem:PSseq}, ensuring that there exists
	 a weak solution of problem \eqref{eq:Problem} in this case (namely, $p<n<p^2$
	 and $q = p$).
\end{proof}

Finally, we turn to consider
the case of $p$-superlinear perburbation $q > p$ (namely,
we aim at proving Theorem \ref{thm:PertNonlinear}). Borrowing the approach used above we need the following preliminary result.

\begin{lem}\label{lem:Path}
	Let $1<p<n$ and $p<q<p^*-1$. Moreover, let
	$$\kappa_{p,q} = \frac{p[q(p-1)+p]}{q(p-1)+p-p(p-1)}.$$
	Then, the following assertions hold.
	\begin{itemize}
		\item[(i)] If $n>\kappa_{p,q}$, then there exists $\varepsilon>0$ small enough such that 
		\begin{equation*}
			\sup_{t \geq 0}\mathcal{J}_{q,\lambda}(t v_{\varepsilon}) < \dfrac{1}{n}(\mathcal{S}_H)^{n/p}\quad
			\text{for all $\lambda > 0$}.
		\end{equation*}
		\item[(ii)] If, instead, $n\leq \kappa_{p,q}$, then there exist $\e > 0$ and $\lambda_0 > 0$ such that
		$$\sup_{t \geq 0}\mathcal{J}_{q,\lambda}(t v_{\varepsilon}) < \dfrac{1}{n}(\mathcal{S}_H)^{n/p}\quad
		\text{for all $\lambda \geq \lambda_0$}.$$
	\end{itemize}
\end{lem}
\begin{proof}
	Our argument is based on a mo\-di\-fi\-ca\-tion 
	of the original approach by Brezis and Nirenberg \cite{BN}. 
		To begin with we observe that, by definition, we have
	\begin{equation} \label{eq:Jtoestimate}
		\begin{split}
			\mathcal{J}_{q,\lambda}(t v_{\varepsilon}) & = 
			\dfrac{t^p}{p}\int_{\Omega}H(\nabla v_{\varepsilon})^p\, dx 
			- \dfrac{t^{p^*}}{p^{*}}\int_\Omega v_\e(x)^{p^*}\,dx-\lambda \dfrac{t^{q}}{q}\int_{\Omega}v_\e(x)^{q}\, dx;
		\end{split}
	\end{equation}
	we then turn to estimate the three integrals in the right-hand side of \eqref{eq:Jtoestimate}.
	\vspace*{0.1cm}
	
	\noindent -\,\,\emph{Estimate of the $L^q$-norm of $v_\e$}. 
	Going back to \eqref{eq:etalower0}, recalling the definition of $\eta_\e$, we get
	\begin{equation} \label{eq:etalower}
		\begin{split}
			\int_{\mathbb{R}^n}v_{\varepsilon}(x)^{q}\, dx \geq c\,\varepsilon^{\frac{q(n-p)}{p^2}-\frac{q(n-p)}{p}+\frac{n(p-1)}{p}},
		\end{split}
	\end{equation}
	provided that $\e > 0$ is sufficiently small (here, $c > 0$ is a constant
	possibly different from line to line but independent of $\e$).
	\vspace*{0.1cm}
	
	\noindent -\,\,\emph{Estimate of the $L^p$-norm of $H(\nabla v_\e)$}. Owing to Lemma
	\ref{lem:talentiane}, and
	arguing exactly as in the proof of Theorem \ref{thm:BN-style}, 
	we immediately get
	\begin{align*}
		\int_{\Omega}H(\nabla v_{\varepsilon})^p\, dx & = 
		(\e^{1/p}c_{n,p})^{n-p}\Big( \dfrac{1}{c_{n,p}^{n-p} \varepsilon^{(n-p)/p}} \mathcal{S}_H^{n/p} + O(1)\Big) \\
		& = \mathcal{S}_H^{n/p} + O(\e^{\frac{n-p}{p}})\qquad\text{as $\e\to 0^+$},
	\end{align*}
	where $\mathcal{S}_H^{n/p}$ is the best Sobolev constant, see \eqref{eq:AnisotropicSobolev}.
	\vspace*{0.1cm}
	
	\noindent -\,\,\emph{Estimate of the $L^{p^*}$-norm of $v_\e$}. Using once again Lemma
	\ref{lem:talentiane}, and we arguing exactly as in the proof of Theorem \ref{thm:BN-style},  
	we obtain
	\begin{align*}
		\int_{\Omega} v_{\varepsilon}^{p^*}(x)\, dx & = 
		(\e^{1/p}c_{n,p})^{n}\Big(\dfrac{1}{c_{n,p}^{{n}}\varepsilon^{n/p}}\mathcal{S}_H^{n/p}
		+ O(1)\Big) \\
		& =  \mathcal{S}_H^{n/p} + O(\e^{n/p})\qquad\text{as $\e\to 0^+$}.
	\end{align*}
	Gathering all these information, from \eqref{eq:Jtoestimate} we then obtain
	\begin{equation} \label{eq:tociteIntro}
		\begin{aligned}
			\mathcal{J}_{q,\lambda}(t v_{\varepsilon})&\leq \dfrac{t^p}{p}
			\big(\mathcal{S}_H^{n/p} + O(\e^{\frac{n-p}{p}})\big) 
			- \dfrac{t^{p^*}}{p^{*}}\big(\mathcal{S}_H^{n/p} + O(\e^{n/p})\big)\\
			& \qquad - 
			c\lambda \dfrac{t^{q}}{q}\varepsilon^{\frac{q(n-p)}{p^2}-\frac{q(n-p)}{p}+\frac{n(p-1)}{p}}\\
			&\leq 
			\dfrac{t^p}{p}
			\big(\mathcal{S}_H^{n/p} + c\,\e^{\frac{n-p}{p}}\big) 
			- \dfrac{t^{p^*}}{p^{*}}\big(\mathcal{S}_H^{n/p} - c\,\e^{n/p}\big) \\ 
			& \qquad - c\lambda \dfrac{t^{q}}{q}\varepsilon^{\frac{q(n-p)}{p}\big(\frac{1}{p}-1\big)+\frac{n(p-1)}{p}}=: g(t),
		\end{aligned}
	\end{equation}
	provided that $\e > 0$ is sufficiently small and for a suitable constant $c > 0$.
	\vspace*{0.1cm}
	
	Inspired also to \cite{GAP2}, to proceed further, we turn to study
	the maximum/maximum points of the function $g$. To this end we first observe that, since
	$$\text{$g(t)\to -\infty$ as $t\to\infty$}, $$
	there exists some point $t_{\varepsilon,\lambda}\geq 0$ such that 
	$$\sup_{t \geq 0}g(t) = g(t_{\varepsilon,\lambda}).$$
	If $t_{\e,\lambda} = 0$, we have $g(t)\leq 0$ for every $t\geq 0$, and the lemma 
	is trivially established as a consequence 
	of \eqref{eq:tociteIntro}. 
	If, instead, $t_{\e,\lambda} > 0$, since $g\in C^1((0,\infty))$ we get
	\begin{align*}
		& 0= g'(t_{\varepsilon,\lambda}) 
		= t_{\varepsilon,\lambda}^{p-1}
		\big(\mathcal{S}_H^{n/p} + c\,\e^{\frac{n-p}{p}}\big)
		- t_{\varepsilon,\lambda}^{p^{*} -1}\big(\mathcal{S}_H^{n/p} - c\,\e^{n/p}\big)
		-
		c\lambda t_{\varepsilon,\lambda}^{q-1}\varepsilon^{\beta_{p,q,n}},
	\end{align*}
	where we have introduced the shorthand notation
	$$\beta_{p,q,n} = \frac{q(n-p)}{p}\Big(\frac{1}{p}-1\Big)+\frac{n(p-1)}{p}.$$
	In particular, recalling that $t_{\varepsilon,\lambda}>0$, we can rewrite the above identity as
	\begin{equation} \label{eq:gprimezerosemp}
		\mathcal{S}_H^{n/p} + c\,\e^{\frac{n-p}{p}} = 
		t_{\varepsilon,\lambda}^{p^{*} -p}\big(\mathcal{S}_H^{n/p} - c\,\e^{n/p}\big)
		+c\lambda t_{\varepsilon,\lambda}^{q-p}\varepsilon^{\beta_{p,q,n}},
	\end{equation}
	from which we easily derive that
	$$t_{\varepsilon,\lambda}< \Big(\frac{\mathcal{S}_H^{n/p} + c\,\e^{\frac{n-p}{p}}}
	{\mathcal{S}_H^{n/p} - c\,\e^{n/p}}\Big)
	^{1/(p^{*}-p)} =: \vartheta_{\e,p}.$$
	We now distinguish two cases, according to the assumptions.
	\medskip
	
	\textsc{Case (i): $n > \kappa_{p,q}$}. In this case we first observe that, since 
	$t_{\e,\lambda} > 0$, from identity 
	\eqref{eq:gprimezerosemp} we easily infer the existence
	of some $\mu_\lambda > 0$ such that
	$$t_{\e,\lambda}\geq \mu_\lambda > 0\quad\text{provided that $\e$ is small enough}.$$
	This, together with the fact that the map
	$$t \mapsto \dfrac{t^p}{p}
	\big(\mathcal{S}_H^{n/p} + c\,\e^{\frac{n-p}{p}}\big) \\
	- \dfrac{t^{p^*}}{p^{*}}\big(\mathcal{S}_H^{n/p} - c\,\e^{n/p}\big)$$
	\noindent is increasing in the closed interval
	$[0,\vartheta_{\e,p}]$, implies
	\begin{equation*}
		\begin{aligned}
			 \sup_{t\geq 0}&g(t) =g(t_{\varepsilon,\lambda}) \\[0.1cm]
			& \qquad < \dfrac{\vartheta_{\e,p}^p}{p}
			\big(\mathcal{S}_H^{n/p} + c\,\e^{\frac{n-p}{p}}\big) - \dfrac{\vartheta_{\e,p}^{p^*}}{p^{*}}\big(\mathcal{S}_H^{n/p} - c\,\e^{n/p}\big) - 
			c\,\varepsilon^{\beta_{p,q,n}}\\[0.1cm]
			& \qquad = \dfrac{1}{n}\frac{(\mathcal{S}_H^{n/p} + c\,\e^{\frac{n-p}{p}})^{\frac{n}{p}}}
			{(\mathcal{S}_H^{n/p} - c\,\e^{n/p})^{\frac{n-p}{p}}}-
			c\,\varepsilon^{\beta_{p,q,n}} \\[0.1cm]
			& \qquad= \frac{1}{n}\mathcal{S}_H^{n/p}(1+c\,\e^{\frac{n-p}{p}})^{\frac{n}{p}}
			(1-c\,\e^{n/p})^{-\frac{n-p}{p}}- c\,\varepsilon^{\beta_{p,q,n}}
			\\[0.1cm]
			& \qquad \leq \dfrac{1}{n}(\mathcal{S}_H)^{n/p}+ c\,\varepsilon^{\frac{n-p}{p}} - 
			c\,\varepsilon^{\beta_{p,q,n}}
			<\dfrac{1}{n}(\mathcal{S}_H)^{n/p},
		\end{aligned}
	\end{equation*}
	\noindent provided that $\e > 0$ is sufficiently small
	(here, $c > 0$ denotes a constant possibly different from line to line
	but independent of $\e$). 
	We explicitly stress that, in the last estimate, we have exploited in a crucial way
	the assumption $n > \kappa_{p,q}$, which is equivalent to
	$$\frac{n-p}{p} > \beta_{n,p,q} = \frac{q(n-p)}{p}\Big(\frac{1}{p}-1\Big)+\frac{n(p-1)}{p}.$$
	
	\textsc{Case (ii):  $p<n\leq \kappa_{p,q}$.} In this second case, we 
	begin by claiming that
	\begin{equation} \label{eq:claimtlambdazero}
		\lim_{\lambda\to\infty}t_{\e,\lambda} = 0.
	\end{equation}
	Indeed, suppose by contradiction that $\ell: = \limsup_{\lambda\to\infty}t_{\e,\lambda} > 0$: 
	then, by possibly choosing a sequence $\{\lambda_k\}_{k}$ diverging to $+\infty$, 
	from \eqref{eq:gprimezerosemp} we get
	\begin{align*}
		\mathcal{S}_H^{n/p} + c\,\e^{\frac{n-p}{p}} = 
		t_{\varepsilon,\lambda_k}^{p^{*} -p}\big(\mathcal{S}_H^{n/p} - c\,\e^{n/p}\big)
		+c\lambda_k t_{\varepsilon,\lambda_k}^{q-p}\varepsilon^{\beta_{p,q,n}}\to+\infty,
	\end{align*}
	but this is clearly absurd. Now we have established 
	\eqref{eq:claimtlambdazero}, we can easily complete the proof of the lemma:
	indeed, by combining \eqref{eq:tociteIntro} with \eqref{eq:claimtlambdazero}, we have
	\begin{align*}
		0 & \leq \sup_{t\geq 0}\mathcal{J}_{q,\lambda}(t \eta_{\varepsilon})
		\leq g(t_{\e,\lambda})
		\\
		& \leq \dfrac{t_{\e,\lambda}^p}{p}
		\big(\mathcal{S}_H^{n/p} + c\,\e^{\frac{n-p}{p}})\big) 
		- \dfrac{t_{\e,\lambda}^{p^*}}{p^{*}}\big(\mathcal{S}_H^{n/p} 
		- c\,\e^{n/p})\big)\to 0\quad\text{as $\lambda\to \infty$},
	\end{align*}
	and this readily implies the existence of $\lambda_0 = \lambda_0(p,n,\e)> 0$ such that
	$$\sup_{t\geq 0}\mathcal{J}_{q,\lambda}(t \eta_{\varepsilon}) < \frac{1}{n}(S_H)^{n/p}\quad\text{for all
		$\lambda\geq \lambda_0$},$$
	provided that $\e > 0$ is small enough but \emph{fixed}. This ends the proof.
\end{proof}

	\begin{proof}[Proof of Theorem \ref{thm:PertNonlinear}] Let $\e > 0$ so small that Lemma \ref{lem:PathLow}
applies, and let $\bar{t} > 0$ be such that
	 $\bar{t}\,\|v_\e\|_{H,p} > \rho$ and $\mathcal{J}_{q,\lambda}(\bar{t}v_\e) < 0$,
	 where $\rho > 0$ is as in Lemma \ref{lem:MPGeometry}. We then set
	\begin{equation*}
	\begin{gathered}
	\alpha = \inf_{\gamma\in \mathcal{C}}\sup_{t\in[0,1]}\mathcal{J}_{q,\lambda}(\gamma(t)),
	 \quad \text{where $\mathcal{C} = \{\gamma\in C([0,1];W_0^{1,p}(\Omega)):\,\gamma(0) =0,\,\gamma(1) 
	 = \bar{t}v_\e\}$.}
	 \end{gathered}
	 \end{equation*}
	 Since $\mathcal{J}_{q,\lambda}$ satisfies the moun\-tain\--pass geometry (see Lemma \ref{lem:MPGeometry}),
	 by the well-known mountain pass theorem we know that $\alpha \geq \delta_0 > 0$, and
	 that there exists a Palais-Smale sequence for $\mathcal{J}_{q,\lambda}$ at level $\alpha$;
	 most importantly, from Lemma \ref{lem:Path} we derive that
	 $$\alpha \leq \max_{t\geq 0}\mathcal{J}_{q,\lambda}(tv_\e)< \frac{\mathcal{S}_H^{n/p}}{n},$$
	 and this holds for every $\lambda > 0$ if $n>\kappa_{p,q}$, or for every $\lambda\geq \lambda_0$,
	 if $n\leq \kappa_{p,q}$.

	 We are then entitled to apply Lemma \ref{lem:PSseq}, ensuring that there exists
	 a weak solution of problem \eqref{eq:Problem} in this case (namely, $1<p<n$
	 and $q > p$).
	\end{proof}
	
\section{Proof of Theorem \ref{thm:Non_Existence}}	 \label{sec:NonExist}
In this last section we provide the proof
of the non-existence result in Theorem \ref{thm:Non_Existence}.
\begin{proof}[Proof of Theorem \ref{thm:Non_Existence}]
  Assume, by contradiction, that there exists a solution
  $u\in C^1(\overline{\Omega})$ of problem \eqref{eq:Problem_general} (for
  some $\lambda\leq 0$). In particular, since $u > 0$ in $\Omega$
  and $u\equiv 0$ pointwise on $\de\Omega$, by the Hopf boundary lemma
  in \cite[Theorem 4.5]{CaRiSc} we have
  \begin{equation} \label{eq:HopfLemmaAnis}
   \text{$\nabla u\not\equiv 0$ on $\de\Omega$}.
  \end{equation}
  Now, since we are assuming $\lambda\leq 0$ (and since $u$ is regular),
  by the anisotropic version of the Pohozaev identity proved in \cite[Theorem 1.2]{MS} we have that $u$ solves 	
	\begin{align}\label{Pohozaev_MS}
		\dfrac{n\lambda}{p}\int_\Omega u^p\, dx + \dfrac{n}{p^\star}\int_\Omega u^{p^\star}\, dx -&\frac{n-p}{p}\int_\Omega H(\nabla u)^p\, dx = -\frac{1}{p}\int_{\partial\Omega}H(\nabla u)^p \langle x,\nu \rangle\, d\sigma \nonumber \\
		& + \int_{\partial\Omega}H(\nabla u)^{p-1}\langle x,\nabla u\rangle \langle \nabla H(\nabla u),\nu \rangle \, d\sigma\, .
	\end{align}
	Since  $u \in C^{1}(\overline{\Omega})$ and $u=0$ on $\partial\Omega$, by
	\eqref{eq:HopfLemmaAnis} we deduce that
	$\nabla u=\langle \nabla u,\nu \rangle\nu$ on $\partial\Omega$;
	hence, using the $0$-homogeneity of $\nabla H$ and Euler theorem we deduce 
	\[
		\begin{split}
			\int_{\partial\Omega}H(\nabla u)^{p-1} \langle x,\nabla u \rangle \langle \nabla H(\nabla u),\nu\rangle\, d\sigma&= \int_{\partial\Omega}H(\nabla u)^{p-1} \langle x,\nabla u \rangle \langle \nabla H(\langle \nabla u,\nu \rangle\nu),\nu\rangle\, d\sigma\\
			& = \int_{\partial\Omega}H(\nabla u)^{p-1} \langle x,\nabla u \rangle \langle \nabla H(\nu),\nu\rangle\, d\sigma\\
			&=\int_{\partial\Omega} H(\nabla u)^p\langle x, \nu \rangle\, d\sigma \, .
		\end{split}
	\]
	Hence, \eqref{Pohozaev_MS} becomes  
	\begin{equation}\label{poho}
	\begin{split}
		& \dfrac{n\lambda}{p}\int_\Omega u^p\, dx + \dfrac{n}{p^\star}\int_\Omega u^{p^\star}\, dx - \frac{n-p}{p}\int_\Omega H(\nabla u)^p\, dx \\
		& \qquad = \frac{p-1}{p}\int_{\partial\Omega}H(\nabla u)^p\langle x, \nu \rangle\, d\sigma \, .
		\end{split}
	\end{equation}
	We now observe that, since $u$
	is a (regular) weak solution
	of 
	problem 
	\eqref{eq:Problem_general}, by taking $v=u$ in \eqref{eq:WeakSolbyParts} we get 
	\begin{equation}\label{test_u}
		\int_\Omega H(\nabla u)^p\, dx= \int_\Omega u^{p^\star}\, dx+\lambda\int_\Omega u^p\, dx\, . 
	\end{equation}
	By substituting \eqref{test_u} in \eqref{poho} we obtain  
	$$
	\lambda\int_\Omega u^p\, dx = \frac{p-1}{p}\int_{\partial\Omega}H(\nabla u)^p\langle x,
	 \nu \rangle\, d\sigma \, .
	$$ 
	On the one hand, since $\lambda\leq 0$ (and since $u > 0$ in $\Omega$), we immediately have 
	$$
	\lambda\int_\Omega u^p\, dx \leq  0\, , 
	$$ 
	while on the other hand, begin $\Omega$ star-shaped, we have  
	$$
	\frac{p-1}{p}\int_{\partial\Omega}H(\nabla u)^p\langle x, \nu \rangle\, d\sigma\geq 0\, ,
	$$
	but this is clearly a contradiction.
\end{proof}

\end{document}